\documentclass[11pt, reqno]{amsart}
\usepackage[section]{bistable}

\title[Uniqueness and multiplicity for elliptic problems in unbounded domains]{Uniqueness and multiplicity for semilinear elliptic problems in unbounded domains}

\author{Henri Berestycki}
\address{%
  HB: Department of Mathematics, University of Maryland, College Park, MD 20742, USA;
  CAMS, École des hautes études en sciences sociales and CNRS, Paris, France;
  Institute for Advanced Study, HKUST, Hong Kong
}
\email{\href{mailto:hb@ehess.fr}{hb@ehess.fr}}

\author{Cole Graham}
\address{CG: Department of Mathematics, University of Wisconsin--Madison, Madison, WI 53706, USA}
\email{\href{mailto:graham@math.wisc.edu}{\tt graham@math.wisc.edu}}

\author{Juncheng Wei}
\address{JW: Department of Mathematics, Chinese University of Hong Kong, Shatin, Hong Kong}
\email{\href{mailto:jcwei@math.ubc.ca}{\tt jcwei@math.ubc.ca}}

\begin{document}

\begin{abstract}
  We study the influence of geometry on semilinear elliptic equations of bistable or nonlinear-field type in unbounded domains.
  We discover a surprising dichotomy between epigraphs that are bounded from below and those that contain a cone of aperture greater than $\pi$: the former admit at most one positive bounded solution, while the latter support infinitely many.
  Nonetheless, we show that every epigraph admits at most one strictly stable solution.
  To prove uniqueness, we strengthen the method of moving planes by decomposing the domain into one region where solutions are stable and another where they enjoy a form of compactness.
  Our construction of many solutions exploits a connection with Delaunay surfaces in differential geometry, and extends to all domains containing a suitably wide cone, including exterior domains.
\end{abstract}

\maketitle

\section{Overview}

We study the number of positive bounded solutions of semilinear elliptic problems in unbounded domains $\Omega \subset \R^d$ with Dirichlet boundary conditions: 
\begin{equation}
  \label{eq:main}
  \begin{cases}
    -\Delta u = f(u) & \text{in } \Omega,\\
    u = 0 & \text{on } \partial \Omega.
  \end{cases}
\end{equation}
This equation describes stationary states of parabolic reaction-diffusion equations with absorbing boundary conditions. 
We consider here two classes of nonlinearities, \emph{bistable} and \emph{field type}.
We first discuss the bistable setting.

Our bistable nonlinearities feature two stable roots and favor one over the other (they are unbalanced).
The prototypical example is $f(u) = u(u - \theta)(1-u)$ for some $\theta \in (0, 1/2)$.
We briefly defer our general hypotheses on $f$ to focus on our main results.

We show that the number of solutions of \eqref{eq:main} depends on the geometry of the domain in a striking fashion.
Consider the epigraph $\Omega = \{y > \phi(x') : x'\in \R^{d-1} \}$ of a uniformly Lipschitz function $\phi \colon \R^{d-1} \to \R$.
We say $\Omega$ is ``bounded from below'' if $\inf \phi > - \infty$.
\begin{theorem}
  \label{thm:unique}
  Let $f$ be a bistable nonlinearity satisfying \ref{hyp:bistable} below.
  If $\Omega$ is a uniformly Lipschitz epigraph that is bounded from below, then \eqref{eq:main} admits a unique positive bounded solution $u$.
  It satisfies $\partial_y u > 0$ and $u(x) \to 1$ uniformly as $\dist(x, \Omega^\cc) \to \infty$.
\end{theorem}
To show this, we use the moving plane method of Alexandrov~\cite{Alexandrov} and Serrin~\cite{Serrin} to establish monotonicity.
This is rather delicate, as the sublevel sets of $\phi$ may be unbounded.
To overcome this obstacle, we employ the ``stable-compact'' framework of the first two authors~\cite{BG23b} through a decomposition of the sublevel sets inspired by arguments of the first author and Nirenberg~\cite{BN91}.
We then use the sliding method to derive uniqueness from monotonicity, as in work of the first author with Caffarelli and Nirenberg~\cite{BCN97b}.

One naturally wonders whether the lower bound on the epigraphs in Theorem~\ref{thm:unique} is necessary.
In fact, this condition is nearly sharp.
Indeed, \eqref{eq:main} admits infinitely many solutions once $\Omega$ contains a circular cone of aperture greater than $\pi$. 
For brevity, we say a domain has aperture larger than $\pi$ if it contains such a cone. 
\begin{theorem}
  \label{thm:multiple}
  Suppose $\Omega$ is a locally Lipschitz domain (not necessarily an epigraph) with aperture greater than $\pi$, 
  and $f$ is bistable and nondegenerate in the sense of \ref{hyp:ground} below.
  Then \eqref{eq:main} admits uncountably many positive bounded solutions.
\end{theorem}
\noindent
The classical nonlinearity $f(u) = u(u - \theta)(1 - u)$ satisfies the nondegeneracy condition \ref{hyp:ground}; see Proposition~\ref{prop:nondegenerate} below.

The multiplicity in Theorem~\ref{thm:multiple} holds in quite general domains, including ``exterior domains'' with compact complement.
Within the class of epigraphs, our results show that the cone $\{y > a \abs{x'}\}$ supports a unique solution when $a \geq 0$ but infinitely many when $a < 0$.

Our construction has its origins in the study of constant mean curvature (CMC) surfaces.
In~\cite{Kapouleas90}, Kapouleas showed how to glue periodic ``Delaunay surfaces'' to produce a rich family of CMC surfaces.
Following formal analogies between differential geometry and elliptic PDEs, Malchiodi~\cite{Malchiodi} adapted these ideas to construct new families of solutions to certain nonlinear field equations in $\R^d$.
These consist of widely-spaced ``ground states'' (see \ref{hyp:ground} below) that formally exert attractive forces on one another.
Malchiodi used a Lyapunov--Schmidt reduction to identify configurations in which all forces balance, producing a solution.

Here, we incorporate Dirichlet boundary into this picture, and find that it plays a major role.
We show that the boundary \emph{repels} ground states.
To attain equilibrium, we must therefore balance this repulsion using the attraction between states.
This places strong geometric constraints on the domain.
When $\Omega$ has aperture greater than $\pi$, we can find many equilibrium configurations of ground states; see Figure~\ref{fig:multiple}.
This is impossible when $\Omega$ is bounded from below, so our construction provides a near-converse to the uniqueness in Theorem~\ref{thm:unique}.
\begin{figure}[t]
  \centering
  \includegraphics[width = 0.9\linewidth]{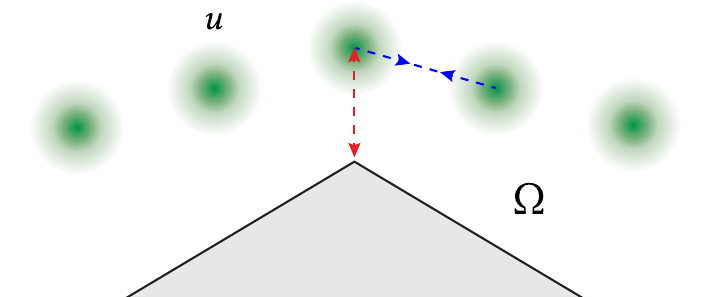}
  \caption{%
    An equilibrium configuration of ground states (green) in a cone $\Omega$ of aperture greater than $\pi$.
    Arrows indicate boundary repulsion (red) and inter-state attraction (blue).
  }
  \label{fig:multiple}
\end{figure}

\medskip

We also treat ``field-type'' nonlinearities resembling $f(u) = -u + u^p$ for a power $p \in (1, \tfrac{d + 2}{d-2})$; for a precise definition see \ref{hyp:field} below.
These are similar to bistable nonlinearities but have no second stable root.
They appear in nonlinear field equations such as the interacting Klein--Gordon equation.

In the bistable case, solutions in Theorem~\ref{thm:unique} converge at infinity to the second stable root $1$ of $f$.
When we remove this root, field equations are left with no bounded positive solutions whatsoever.
\begin{theorem}
  \label{thm:field-unique}
  Let $f$ be field-type as defined in \ref{hyp:field} below.
  If $\Omega$ is a uniformly Lipschitz epigraph that is bounded from below, then \eqref{eq:main} has no positive bounded solution.
\end{theorem}
On the other hand, for the motivating nonlinearity $-u + u^p$, domains with aperture exceeding $\pi$ admit infinitely many solutions.
\begin{theorem}
  \label{thm:field-multiple}
  Suppose $\Omega$ is a domain with aperture larger than $\pi$, and $f(u)  = -u + u^p$ for $p \in (1, \tfrac{d + 2}{d-2})$.
  Then \eqref{eq:main} admits uncountably many positive bounded solutions.
\end{theorem}
We now consider the stability of solutions in epigraphs.
We show that every (slightly more regular) epigraph has at most one \emph{strictly stable} solution, even when it admits many others.
\begin{theorem}
  \label{thm:stable}
  Let $\Omega$ be a uniformly $\m{C}^{1,\al}$ epigraph for some $\al \in (0, 1)$.
  If $f$ is bistable~\ref{hyp:bistable}, then \eqref{eq:main} admits exactly one positive bounded strictly stable solution.
  If $f$ is field-type~\ref{hyp:field}, then \eqref{eq:main} has no such solution.
\end{theorem}
By ``strictly stable,'' we mean the Dirichlet principal eigenvalue $\lambda$ of the operator $-\Delta - f'(u)$ in $\Omega$ is positive; see \eqref{eq:eigenvalue} below.
In the bistable case, it is relatively straightforward to construct a candidate solution satisfying $\lambda \geq 0$.
However, establishing strict positivity in an unbounded domain is nontrivial, and we again deploy the stable-compact method of~\cite{BG23b} to ``localize'' the problem.

We emphasize that Theorem~\ref{thm:stable} does not require $\Omega$ to be bounded from below, so \eqref{eq:main} may well have other (unstable) solutions.
Indeed, the solutions constructed in Theorem~\ref{thm:multiple} and \ref{thm:field-multiple} are strictly unstable (see Proposition~\ref{prop:unstable} below).

\subsection*{Hypotheses and notation}
In all the following, we assume our epigraphs are uniformly Lipschitz in the sense that the defining function $\phi$ is globally Lipschitz.
Similarly, we say an epigraph is uniformly $\m{C}^{1,\al}$ if $\op{Lip} \phi < \infty$ and $\norm{\nab \phi}_{\m{C}^\al(\R^d)} < \infty$.
When we study more general domains, we merely assume that $\partial\Omega$ is locally Lipschitz so that solutions $u$ of \eqref{eq:main} are continuous up to the boundary.

Throughout, we assume that $f \in \m{C}_{\text{loc}}^{1,\gamma}([0, \infty); \R)$ for some $\gamma \in (0, 1)$, $f(0) = 0$, and $f'(0) < 0$.
We say $f$ is (unbalanced) \emph{bistable} if in addition
\begin{enumerate}[label = \textup{(B)}]
\item
  There exists $\theta \in (0, 1)$ such that $f(\theta) = f(1) = 0$, $f'(\theta) > 0$, $f'(1) < 0$, $f|_{(0, \theta)} < 0,$ $f|_{(\theta, 1)} > 0$, $f|_{(1, \infty)} < 0$, and $\int_0^1 f > 0$.
  \label{hyp:bistable}
\end{enumerate}
Such nonlinearities arise widely in ecology and materials science.
Note that the Allen--Cahn equation, with its symmetric double-well potential, involves a balanced nonlinearity, and thus falls outside our results.

Next, we say $f$ is \emph{field-type} if
\begin{enumerate}[label = \textup{(F)}]
\item There exists $\theta > 0$ such that $f(\theta) = 0$, $f'(\theta) > 0$, $f|_{(0, \theta)} < 0$, and $f|_{(\theta, \infty)} > 0$.
  \label{hyp:field}
\end{enumerate}
This captures a wide range of nonlinear field equations, including $f(u) = -u + u^p$ for $p > 1$.
We depict both types of nonlinearity in Figure~\ref{fig:nonlinearities}.
\begin{figure}[t]
  \centering
  \includegraphics[width = 0.75\linewidth]{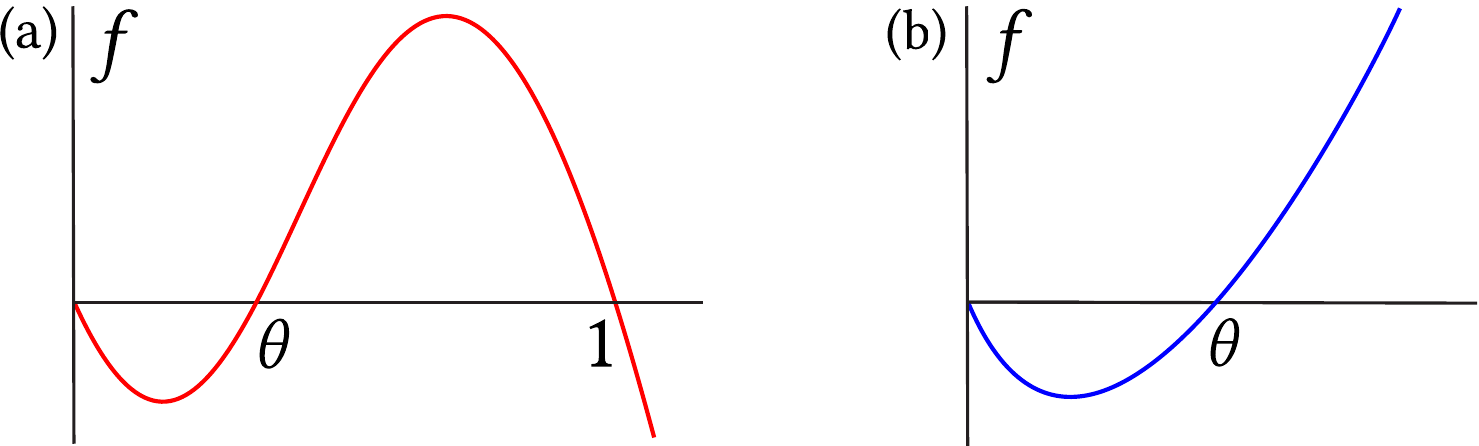}
  \caption{(a) Bistable and (b) field-type nonlinearities.}
  \label{fig:nonlinearities}
\end{figure}

Some of our results further require the existence of a nondegenerate ground state. We express this through the following condition.
\begin{enumerate}[label = \textup{(N)}]
\item
  \label{hyp:ground}
  The equation $-\Delta u = f(u)$ has a positive radially symmetric solution $U$ in $\R^d$ that decays exponentially at infinity.
  Moreover, $U$ is nondegenerate:
  \begin{equation}
    \label{eq:nondegenerate}
    \op{ker}(-\Delta - f' \circ U) \cap L^\infty = \op{span}\{\partial_1 U, \ldots, \partial_d U\}.
  \end{equation}
\end{enumerate}
Several model nonlinearities satisfy this hypothesis, including the field-type $-u + u^p$ for $p \in (1, \frac{d + 2}{d - 2})$ and the bistable $u(u - \theta)(1 - u)$ for $\theta \in (0, 1/2)$.
We collect a number of results ensuring \ref{hyp:ground} in Appendix~\ref{sec:nondegenerate}.

Our arguments make frequent use of the generalized principal eigenvalue of a linear elliptic operator $\m{L}$ on a domain $V$ with Dirichlet boundary conditions.
This notion, introduced by the first author in works with Nirenberg and Varadhan~\cite{BNV} and Rossi~\cite{BR}, is defined by the expression
\begin{equation}
  \label{eq:eigenvalue}
  \lambda(-\m{L}, V) \coloneqq \sup\big\{\lambda : \exists \psi \in W_{\text{loc}}^{2,d}(V) \text{ s.t. } \psi > 0, \, (\m{L} + \lambda) \psi \leq 0\big\}.
\end{equation}

\subsection*{Related works}
There is a vast literature on qualitative properties of semilinear elliptic equations.
Here we discuss a handful of works that are particularly relevant to our study.
\smallskip

The whole space has long been known to admit large families of solutions.
Of these, stable and monotone solutions are of particular importance.
In balanced bistable (Allen--Cahn) equations, de Giorgi famously conjectured that monotone solutions are one-dimensional up to dimension 8~\cite{deGiorgi}.
This insightful conjecture has been largely confirmed in a sequence of remarkable works~\mbox{\cite{GG,AC,Savin}}, while del Pino, Kowalczyk, and the third author have shown the converse in dimensions larger than 8~\cite{dPKW}.
For \emph{unbalanced} nonlinearities like those studied here, Liu, Wang, Wu, and the third author have shown that every stable solution is constant regardless of the dimension~\cite{LWWW}.
We put this result to frequent use here.

A great deal of attention has also focused on ``ground states:'' positive solutions in $\R^d$ that decay at infinity.
As indicated above, these play an important role in the present work.
Using a variational characterization, the first author and Lions constructed ground states for quite general equations~\cite{BL}.
Due to the seminal work of Gidas, Ni, and Nirenberg~\cite{GNN79}, these solutions are radially symmetric.
Using this symmetry, Peletier and Serrin~\cite{PS} and Kwong~\cite{Kwong} proved the uniqueness of ground states for certain bistable and nonlinear field equations, respectively.
\smallskip

The introduction of Dirichlet boundary can dramatically restrict the set of solutions in epigraphs.
This was first observed by Esteban and Lions~\cite{EL}, who ruled out ground states in coercive epigraphs.
Exchanging coercivity for positivity in the nonlinearity, Caffarelli, Nirenberg, and the first author showed the uniqueness of positive bounded solutions in uniformly Lipschitz epigraphs~\cite{BCN97b}.
The first two authors have since relaxed the need for a global Lipschitz bound~\cite{BG23b}.
After completing the present work, we learned of a simultaneous effort~\cite{BFS} that demonstrates the monotonicity of solutions in epigraphs that are bounded from below.
\smallskip

In the special case of the half-space, the first two authors have proved the uniqueness of solutions for several classes of nonlinearity~\cite{BG22}.
Dupaigne and Farina~\cite{DupFar} have likewise classified solutions with nonnegative nonlinearity in sufficiently low dimensions.
The recent work~\cite{LWWW} of the third author with Liu, Wang, and Wu establishes uniqueness in all dimensions for quite general nonlinearities.
\smallskip

In the present paper, we focus on bistable and field-type nonlinearities.
Another widely-studied class consists of nonnegative, convex, superlinear nonlinearities, exemplified by the Lane--Emden equation with $f(u) = u^p$.
In this context, Cabr\'e, Figalli, Ros-Oton, and Serra~\cite{CFROS} recently resolved a well-known conjecture of Brezis, showing that stable but possibly singular solutions of \eqref{eq:main} in bounded domains are in fact smooth provided $d \leq 9$ (which is sharp).
For the Lane--Emden equation, Farina~\cite{Farina} has established a number of results linking uniqueness and stability in unbounded domains subject to (often sharp) conditions on the power $p$ and the dimension $d$.
Despite their superficial similarity, we note that the Lane--Emden equation $f(u) = u^p$ and the nonlinear-field equation $f(u) = -u + u^p$ behave quite differently due to contrasting linearizations at zero.

\subsection*{Organization}
We establish uniqueness in epigraphs that are bounded from below (Theorems~\ref{thm:unique} and \ref{thm:field-unique}) in Section~\ref{sec:unique}.
In Section~\ref{sec:multiple}, we construct many solutions in domains with aperture greater than $\pi$, and thus prove Theorems~\ref{thm:multiple} and \ref{thm:field-multiple}.
We study the linear stability of solutions and prove Theorem~\ref{thm:stable} in Section~\ref{sec:stable}.
We collect known results regarding the nondegeneracy of ground states in Appendix~\ref{sec:nondegenerate}.

\section*{Acknowledgments}
HB has received funding from the French ANR project ANR-23-CEE40-0023-01 ReaCh.
CG is partially supported by the National Science Foundation under DMS-2406946.
JW is partially supported by the General Research Fund of Hong Kong through ``New frontiers in singularity formations of nonlinear partial differential equations.''

\section{Uniqueness}
\label{sec:unique}

In this section, we show uniqueness results leading to Theorems~\ref{thm:unique} and \ref{thm:field-unique}.
Throughout, we assume that $\Omega$ is a uniformly Lipschitz epigraph that is bounded from below, so $\op{Lip} \phi < \infty$ and $\inf \phi > -\infty$.
We are then free to shift $\Omega$ so that $\inf \phi = 0$.
Our approach relies on the continuity of the principal eigenvalue with respect to the potential:
\begin{lemma}
  \label{lem:eig-stable}
  Let $q(1) = 1$, $q(2) = 3/2$, and $q(d) \coloneqq d/2$ for $d \geq 3$.
  There exists $C(d) \geq 1$ such that if $\norm{V}_{L^q} \leq C^{-1}$, then
  \begin{equation*}
    -C \norm{V}_{L^q} \leq \lambda(-\Delta + V, B_1) - \lambda(-\Delta, B_1) \leq C \norm{V}_{L^1}.
  \end{equation*}
\end{lemma}
In fact, $q(2)$ can assume any value greater than $1$, but the constant $C$ depends on its value.
We will not need this refinement.
\begin{proof}
  Given a potential $W \colon B_1 \to \R,$ let $\lambda_W \coloneqq \lambda(-\Delta + W, B_1)$.
  Let $\phi_W > 0$ denote the corresponding principal eigenfunction satisfying $\norm{\phi_W}_{L^2} = 1$.
  We use the Rayleigh quotient to control $\lambda_V - \lambda_0$.

  First, using $\phi_0$ as a test function for $\lambda_V$, we see that in all dimensions,
  \begin{equation*}
    \lambda_V - \lambda_0 \leq \int_{B_1} V \phi_0^2 \leq \norm{\phi_0}_{L^\infty}^2 \norm{V}_{L^1} \eqqcolon C_1(d) \norm{V}_{L^1}.
  \end{equation*}
  
  Next, using $\phi_V$ as a test function for $\lambda_0,$ we find $\lambda_0 \leq \int \abs{\nab\phi_V}^2 \eqqcolon D$.
  Assume for the moment that $d \geq 3$.
  Then the Sobolev inequality yields $\norm{\phi_V}_{L^{2_*}}^2 \leq C_2(d) D$ for $2_* = \frac{2d}{d-2} \in (2, \infty)$.
  It follows from H\"older that
  \begin{equation}
    \label{eq:Holder}
    \lambda_V = D + \int_{B_1} V \phi_V^2 \geq D - \norm{V}_{L^{d/2}} \norm{\phi_V}_{L^{2_*}}^2 \geq D(1 - C_2 \norm{V}_{L^{d/2}}).
  \end{equation}
  Assume $\norm{V}_{L^{d/2}} \leq C_2^{-1}$.
  Using $\lambda_0 \leq D$ in \eqref{eq:Holder}, we find $\lambda_V \geq (1 - C_2\norm{V}_{L^{d/2}}) \lambda_0.$
  That is,
  \begin{equation*}
    \lambda_V - \lambda_0 \geq -C_2 \lambda_0 \norm{V}_{L^{d/2}} \eqqcolon -C_3(d) \norm{V}_{L^{d/2}}.
  \end{equation*}
  The low-dimensional statements follow from similar reasoning.
\end{proof}
We use this lemma in the following form:
\begin{corollary}
  \label{cor:thin}
  For all $a \in \R,$ $\mathbf{e} \in S^{d-1}$, and $\eps > 0$, there exists $\eta(\abs{a}, \eps, d) > 0$ such that for every open $S \subset \R^d$ satisfying $\m{H}^1(S \cap \ell) \leq \eta$ for all lines $\ell \subset \R^d$ parallel to $\mathbf{e}$,
  \begin{equation*}
    \abss{\lambda(-\Delta + a \tbf{1}_{S}, \R^d)} \leq \eps.
  \end{equation*}
\end{corollary}
\begin{proof}
  By a variant \cite[Theorem~3.5]{BG23b} of Lieb's localization inequality~\cite{Lieb},
  \begin{equation*}
    \absb{\inf_{x \in \R^d} \lambda\big(\!-\Delta + a \tbf{1}_S, B_R(x)\big) - \lambda(-\Delta + a \tbf{1}_S, \R^d)} \leq \al R^{-2}
  \end{equation*}
  where $\al(d) \coloneqq \lambda(-\Delta, B_1) > 0$.
  Choose $R(\eps, d) \coloneqq \sqrt{3\al/\eps}$, so $\al R^{-2} = \eps/3$.

  We can scale Lemma~\ref{lem:eig-stable} to produce $q(d) \in [1, \infty)$ and $C(\eps, d) > 1$ such that
  \begin{equation*}
    \abs{\lambda\big(\!-\Delta + a \tbf{1}_S, B_R(x)\big) - \lambda(-\Delta, B_R)} \leq C \norm{a \tbf{1}_S}_{L^{q}(B_R(x))} = C \abs{a} \abs{S \cap B_R(x)}^{1/q}
  \end{equation*}
  provided $\abs{a} \abs{S \cap B_R(x)}^{1/q} \leq C^{-1}$.
  Now Fubini implies that $\abs{S \cap B_R(x)} \leq \abss{B_R^{d-1}} \eta$, where $B_R^{d-1}$ denotes the ball of radius $R$ in $\R^{d-1}$.
  Recalling that $R = R(\eps, d)$, we fix $\eta(\abs{a}, \eps, d) > 0$ such that
  \begin{equation*}
    C \abs{a} \big(\abss{B_R^{d-1}} \eta\big)^{1/q} = \min\{\eps/3, 1\}.
  \end{equation*}
  Then we have
  \begin{equation*}
    \abs{\lambda\big(\!-\Delta + a \tbf{1}_S, B_R(x)\big) - \lambda(-\Delta, B_R)} \leq \frac{\eps}{3}.
  \end{equation*}
  Finally, $\lambda(-\Delta, B_R) = \al R^{-2} = \eps/3$, so the corollary follows from the triangle inequality.
\end{proof}
Next, we apply the stable-compact framework of~\cite{BG23b} to the method of moving planes~\cite{Alexandrov,Serrin} to show that solutions of \eqref{eq:main} are monotone in $y$.
\begin{proposition}
  \label{prop:monotone}
  Suppose $f$ is bistable \ref{hyp:bistable} or field-type \ref{hyp:field}.
  Then every positive bounded solution $u$ of \eqref{eq:main} satisfies $\partial_y u > 0$.
\end{proposition}
\noindent
This is the only component of the proof that requires $\Omega$ to be bounded from below.
\begin{proof}
  Given $\mu > 0$, define the sub-level set
  \begin{equation*}
    \Sigma^\mu \coloneqq \{(x', y) \in \Omega : y < \mu\}.
  \end{equation*}
  In $\Sigma^\mu$, define the reflection $v^\mu(x', y) \coloneqq u(x', 2\mu - y)$ and the difference $w^\mu \coloneqq v^\mu - u$.
  This satisfies the linear equation
  \begin{equation}
    \label{eq:diff}
    \begin{cases}
      -\Delta w^\mu = q^\mu w^\mu & \text{in } \Sigma^\mu,\\
      w^\mu \gneqq 0 & \text{on } \partial \Sigma^\mu,
    \end{cases}
  \end{equation}
  for $q^\mu \coloneqq \frac{f(v^\mu) - f(u)}{v^\mu - u} \tbf{1}_{v^\mu \neq u}$, which is bounded because $f$ is Lipschitz.
  The boundary condition in \eqref{eq:diff} holds because $w^\mu = 0$ on $\{y = \mu\}$ and $v^\mu > 0 = u$ on $\partial \Omega \cap \{y < \mu\}$.
  It suffices to show that $w^\mu \geq 0$ for all $\mu > 0$, for then $u$ is increasing in $y$ and $\partial_y u > 0$ by the strong maximum principle.
  
  Because $\Omega$ is bounded from below by $\{y = 0\}$, $\Sigma^\mu$ is contained in the slab $\R^{d-1} \times (0, \mu)$.
  It follows from Lemma~3.2 of~\cite{BG23b} that
  \begin{equation*}
    \lambda(-\Delta - q^\mu, \Sigma^\mu) \geq \lambda\big(\!-\Delta, \R^{d-1} \times (0, \mu)\big) - \Lip f = \pi^2 \mu^{-2} - \Lip f.
  \end{equation*}
  Hence when $\mu$ is sufficiently small, $\lambda(-\Delta - q^\mu, \Sigma^\mu) > 0$, \eqref{eq:diff} satisfies the maximum principle, and $w^\mu \geq 0$.

  Define
  \begin{equation}
    \label{eq:mu-crit}
    \mus \coloneqq \inf\{\mu > 0 : w^\mu \not\geq 0 \text{ in } \Sigma^\mu\}.
  \end{equation}
  By the above argument, $\mus > 0.$
  Suppose for the sake of contradiction that $\mus < \infty.$
  By continuity, we have $w^\mus \geq 0$ in $\Sigma^\mus$.

  Because $f$ is $\m{C}^1$ and $f'(0) < 0$, there exists $s > 0$ such that
  \begin{equation*}
    \sup_{[0, s]} f' \leq f'(0)/2 < 0.
  \end{equation*}
  We apply Corollary~\ref{cor:thin} with $a = -2 \Lip f$, $\tbf{e} = e_y$, and $\eps = \abs{f'(0)}/4$ to obtain a length $\eta(a, \eps, d) > 0$.
  We employ the corresponding eigenvalue bound below.
  
  Schauder estimates up to the boundary (e.g., \cite[Lemma~1.2.4]{Kenig}) imply that $u$ is uniformly continuous.
  Hence there exists $\delta \in (0, \eta/3]$ such that $\abs{u(x) - u(z)} \leq s$ for all $x, z \in \Omega$ such that $\abs{x - z} \leq 2\delta$.
  
  We define the ``compact part''
  \begin{equation}
    \label{eq:compact}
    \Sigma_- \coloneqq \big\{(x', y) \in \Sigma^\mus : u(x', y) \geq s \enspace\text{and}\enspace \min\{y - \phi(x'), \mus - y\} \geq \delta\big\}.
  \end{equation}
  Loosely, $\Sigma_-$ is the set $\delta$-isolated from $\partial \Sigma^\mus$ where $u$ is moderately large.
  
  We claim that $\inf_{\Sigma_-} w^\mus > 0$.
  To see this, suppose to the contrary that there exists a sequence $x_n = (x_n', y_n) \in \Sigma_-$ such that $w^\mus(x_n) \to 0$ as $n \to \infty$.
  We horizontally recenter and define
  \begin{equation*}
    u_n(x', y) \coloneqq u(x' + x_n', y) \And \phi_n(x') \coloneqq \phi(x' + x_n').
  \end{equation*}
  We observe that $y_n$ lies in the compact interval $[\delta, \mus - \delta]$.
  Due to the uniform regularity of $u$ and $\phi$, we can thus extract a subsequential limit $(y_\infty, u_\infty, \phi_\infty)$ of $(y_n, u_n, \phi_n)$ as $n \to \infty$.
  We accordingly construct $\Omega_\infty$, $\Sigma_\infty^\mus$, $v_\infty^\mus$, and $w_\infty^\mus$.
  
  The definition of $\Sigma_-$ implies that $u_\infty(0, y_\infty) \geq s$ and $(0, y_\infty)$ is $\delta$-separated from $\partial \Sigma_\infty^\mus$; in particular, $(0, y_\infty) \in \Sigma_\infty^\mus$.
  Next, interior Schauder estimates imply that $u_\infty$ and $w_\infty^\mus$ satisfy analogues of the elliptic equations \eqref{eq:main} and \eqref{eq:diff}, respectively.
  By the strong maximum principle, $u_\infty > 0$ in $\Omega_\infty$.
  Hence $v_\infty^\mus > 0$ on $\partial \Omega_\infty \cap \{y < \mus\}$, and a second application of the strong maximum principle yields $w_\infty^\mus > 0$ in $\Sigma_\infty^\mus$.
  However, the definition of the sequence $(x_n)_n$ implies that $w_\infty^\mus(0, y_\infty) = 0$, contradicting $(0, y_\infty) \in \Sigma_\infty^\mus$.
  So indeed $\inf_{\Sigma_-} w_\mus > 0$.

  Given $\mu \in [\mus, \mus + \delta)$, we define the ``stable part'' $\Sigma_+^\mu \coloneqq \Sigma^\mu \setminus \Sigma_-$.
  We emphasize that $\Sigma_+^\mu$ varies in $\mu$ while $\Sigma_-$ does not.
  Next, let
  \begin{equation*}
    S \coloneqq \big\{(x', y) \in \Sigma^\mu : y - \phi(x') \leq \delta \text{ or } \mu - y \leq 2 \delta\big\}.
  \end{equation*}
  We observe that $\m{H}^1(S \cap \ell) \leq 3 \delta \leq \eta$ for all lines $\ell$ parallel to $e_y$.
  That is, $S$ satisfies the hypothesis of Corollary~\ref{cor:thin}, so by our choice of $\eta$,
  \begin{equation}
    \label{eq:wiggle}
    \absb{\lambda\big(\!-\Delta -2 \Lip f \,\tbf{1}_S, \, \R^d\big)} \leq \frac{\abs{f'(0)}}{4}.
  \end{equation}
  Next, because $\mu - \mus < \delta$,
  \begin{equation*}
    S \supset \big\{(x', y) \in \Sigma^\mus : \min\{y - \phi(x'), \mus - y\} < \delta\big\}.
  \end{equation*}
  From the definition of $\Sigma_\pm$, we see that $\Sigma_+^\mu \setminus S \subset \Sigma^\mus$ and $u \leq s$ in $\Sigma_+^\mu \setminus S$.

  We claim that $w^\mu$ satisfies the following maximum principle: if $w^\mu \geq 0$ on $\partial \Sigma_+^\mu$, then $w^\mu \geq 0$ in $\Sigma_+^\mu$.
  To see this, suppose $w^\mu|_{\partial \Sigma_+^\mu} \geq 0$ and consider the contrary set $Q \coloneqq \{w^\mu < 0\} \cap \Sigma_+^\mu$.
  The boundary hypothesis implies that $w = 0$ on $\partial Q$.
  Also, on $Q \setminus S$ we have $v^\mu < u \leq s$.
  Thus by the mean value theorem and the definition of $s$,
  \begin{equation*}
    q^\mu = \frac{f(v^\mu) - f(u)}{v^\mu - u} \leq \frac{f'(0)}{2} \quad \text{in } Q \setminus S.
  \end{equation*}
  Combining this with \eqref{eq:wiggle}, we obtain
  \begin{align*}
    \lambda(-\Delta - q^\mu, Q) &\geq \lambda\Big(\!-\Delta + \frac{\abs{f'(0)}}{2} - 2 \Lip f \, \tbf{1}_S, Q\Big)\\
                                &\geq \frac{\abs{f'(0)}}{2} + \lambda(-\Delta - 2 \Lip f \, \tbf{1}_S, \R^d) \geq \frac{\abs{f'(0)}}{4}.
  \end{align*}
  In particular, $\lambda(-\Delta - q^\mu, Q) > 0$, so by the maximum principle $w^\mu|_Q = 0$.
  That is, $Q$ is in fact empty and $w^\mu \geq 0$ in $\Sigma_+^\mu$, as claimed (still under the hypothesis $w^\mu|_{\partial \Sigma_+^\mu} \geq 0$).
  
  Now recall that $\inf_{\Sigma_-} w^\mus > 0$.
  The uniform continuity of $u$ implies that $w^\mu$ is uniformly continuous in $\mu$.
  Thus there exists $\eps \in (0, \delta]$ such that for all $\mu \in [\mus, \mus + \eps)$, $\inf_{\Sigma_-} w^\mu \geq 0$ and hence $w^\mu|_{\partial \Sigma_+^\mu} \geq 0$.
  By the maximum principle shown above, $w^\mu \geq 0$ in $\Sigma_+^\mu$ as well.
  Thus $w^\mu \geq 0$ in $\Sigma^\mu = \Sigma_- \cup \Sigma_+^\mu$ for all $\mu \in [\mus, \mus + \eps)$, contradicting the definition of $\mus$.
  So in fact $\mus = \infty$ and $w^\mu \geq 0$ always; this completes the proof.
\end{proof}
We next record the upper semicontinuity of the generalized principle eigenvalue in the domain and potential.
Similar results have appeared before, for example Proposition~5.3 in~\cite{BHRossi} and Lemma~2.4 of~\cite{BG25}.
We unify these and extend them to a lower-regularity setting.
We recall the \emph{connected} limit of a sequence of domains from Definitions~2.3 and A.3 of \cite{BG25}, which extracts a connected component from a sequence of domains converging locally uniformly.
These are formulated for uniformly $\m{C}^{2,\al}$ domains, but the definition readily extends to uniformly Lipschitz domains: we simply ask the relevant functions to converge in $\m{C}^\al$ for some $\al < 1$.
\begin{lemma}
  \label{lem:semicontinuous}
  Consider a sequence of domains $\Omega_n$ and potentials $V_n \colon \Omega_n \to \R$ such that $(\Omega_n)_n$ is uniformly Lipschitz and $V_n$ is uniformly $\m{C}^\gamma$ for some $\gamma \in (0, 1)$.
  If $(\Omega_*, V_*)$ is the connected limit of $(\Omega_n, V_n)_n$, then
  \begin{equation*}
    \lambda(-\Delta + V_*, \Omega_*) \geq \lim_{n \to \infty} \lambda(-\Delta + V_n, \Omega_n).
  \end{equation*}
\end{lemma}
\noindent
We closely follow the proof of Lemma~2.4 in~\cite{BG25}.
\begin{proof}
  Up to a shift, we are free to assume that $0 \in \bigcap_n \Omega_n \cap \Omega_*$.
  We claim that each domain $\Omega_n$ admits a principal eigenfunction $\varphi_n > 0$ normalized by $\varphi_n(0) = 1$.
  Indeed, for each $R > 0,$ let $\Omega_n^R$ denote the connected component of $\Omega_n \cap B_R$ containing $0$.
  Let $\lambda_n^R \coloneqq \lambda(-\Delta + V_n, \Omega_n^R)$ and analogously define $\lambda_n$ and $\lambda_*$.
  By Theorem~2.1 of \cite{BNV}, there exists a principal eigenfunction $\varphi_n^R > 0$ satisfying $(-\Delta + V_n) \varphi_n^R = \lambda_n^R \varphi_n^R$, $\varphi_n^R = 0$ on $\partial\Omega_n^R$, and $\varphi_n^R(0) = 1$.
  Lemma~1.2.4 of \cite{Kenig} implies that $\varphi_n^R \in \m{C}^\beta(\bar{\Omega}{}_n^R)$ for some $\beta \in (0, 1)$ depending on the Lipschitz constant of $\Omega_n$.
  In fact, the proof implies that for any fixed compact $K \subset \Omega_n$ and $R$ sufficiently large, $\norms{\varphi_n^R}_{\m{C}^\beta(K)}$ is uniformly bounded in $R$.
  Sending $R \to \infty$ and using Arzel\`a--Ascoli and diagonalization, we can extract a subsequence of radii $R$ along which $\varphi_n^R$ converges in $\m{C}_{\text{loc}}^{\beta'}(\bar{\Omega}_n)$ for any $\beta' \in (0, \beta)$ to a limit $\varphi_n \in \m{C}^\beta(\bar{\Omega}_n)$.
  Interior and boundary estimates imply that $\varphi_n > 0$ satisfies $(-\Delta + V_n)\varphi_n = (\lim_{R \to \infty } \lambda_n^R) \varphi_n$, $\varphi_n|_{\partial\Omega_n} = 0$, and $\varphi_n(0) = 1$.
  Using $\varphi_n$ as a witness in \eqref{eq:eigenvalue}, we see that $\lambda_n \geq \lim_{R \to \infty } \lambda_n^R$.
  On the other hand, \eqref{eq:eigenvalue} easily shows that the principal eigenvalue is decreasing in the domain (\emph{Cf.} \cite[Proposition~2.3(iii)]{BR}), so $\lambda_n \leq \lambda_n^R$ for all $R$.
  We conclude that $\lambda_n = \lim_{R \to \infty } \lambda_n^R$, so indeed $\lambda_n$ has a corresponding positive normalized eigenfunction.
  
  We now perform the same procedure in the limit $n \to \infty$.
  Using the locally uniform $\m{C}^\beta$ estimates from above, we can extract a subsequential limit $\varphi_* \colon \Omega_* \to \R_+$ of $\varphi_n$ as $n \to \infty$.
  Interior estimates again imply that $\varphi_*$ satisfies $(-\Delta + V_*) \varphi_* = (\lim_{n \to \infty} \lambda_n) \varphi_*$.
  Using $\varphi_*$ as a witness in \eqref{eq:eigenvalue}, we see that $\lambda_* \geq \lim_{n \to \infty} \lambda_n$, as desired.
\end{proof}
We can now identify the limiting behavior of solutions far from the boundary.
\begin{lemma}
  \label{lem:limit}
  Let $\s{U}$ denote the set of positive bounded solutions of \eqref{eq:main}.
  Then if $f$ is bistable~\ref{hyp:bistable},
  \begin{equation}
    \label{eq:uniform}
    \lim_{y' \to \infty} \inf_{\substack{u \in \s{U}\\x' \in \R^{d-1}}} u\big(x', y' + \phi(x')\big) = 1.
  \end{equation}
  If $f$ is field-type~\ref{hyp:field}, then $\s{U} = \emptyset$.
\end{lemma}
\noindent
That is, when $f$ is bistable $u \to 1$ uniformly as the vertical distance to the boundary tends to infinity.
When $f$ is field-type, there are no positive bounded solutions.
This completes the proof of Theorem~\ref{thm:field-unique} for field-type nonlinearities.
\begin{proof}
  Let $u$ be a positive bounded solution of \eqref{eq:main}.
  By Proposition~\ref{prop:monotone}, $\partial_y u > 0$.
  It follows that $u(x', y)$ converges to a positive bounded limit $u_\infty(x')$ as $y \to \infty$.
  Interior Schauder estimates imply that this convergence is locally uniform in $x'$ and $u_\infty$ solves $-\Delta u_\infty = f(u_\infty)$ in $\R^{d-1}$.
  Moreover, differentiating \eqref{eq:main} in $y$, we see that $\partial_y u > 0$ solves its linearization about $u$.
  By \eqref{eq:eigenvalue}, $\lambda(-\Delta-f'(u), \Omega) \geq 0$.

  Since $f'$ is continuous, $\big(\R^d, f'(u_\infty)\big)$ is the connected limit of $\big(\Omega, f'(u)\big)$ as $y \to \infty$ in the sense described before Lemma~\ref{lem:semicontinuous}.
  By Lemma~\ref{lem:semicontinuous} we have
  \begin{equation*}
    \lambda(-\Delta - f'(u_\infty), \R^d) \geq \lambda(-\Delta-f'(u), \Omega) \geq 0.
  \end{equation*}
  Furthermore, Lemma~3.2 of~\cite{BG23b} allows us to drop the free $y$-dimension:
  \begin{equation*}
    \lambda(-\Delta - f'(u_\infty), \R^{d-1})  = \lambda(-\Delta - f'(u_\infty), \R^d) \geq 0.
  \end{equation*}
  That is, $u_\infty$ is a (weakly) stable solution of $-\Delta v = f(v)$ in $\R^{d-1}$.
  By Theorem~1.4 of~\cite{LWWW}, $u_\infty$ is a positive constant, which must hence be a weakly stable root of $f$.
  If $f$ is bistable, then $u_\infty = 1$.
  On the other hand, field-type nonlinearities have no positive weakly stable root, so in fact $u$ does not exist.

  Assuming $f$ is bistable, we now show that the convergence $u \to 1$ is uniform in the sense of \eqref{eq:uniform}.
  Fix $\eps \in (0, 1)$.
  Following \cite[Proposition~2.2]{BG22}, we can construct a subsolution $w \leq 1 - \eps$ of $-\Delta v = f(v)$ such that $w(0) = 1 - \eps$ and $\supp w = B_R$ for some $R(\eps) > 0$.
  For any $x_0' \in \R^{d-1}$, the locally uniform convergence $u(x', y) \to 1$ implies that $v(x' - x_0', y - h) \leq u$ for $h \gg 1$.
  Continuously reducing $h$, the strong maximum principle ensures that this continues to hold so long as $B_R(x_0', h) \subset \Omega$.
  A simple geometric argument shows that this holds whenever $h - \phi(x_0') \geq R \sqrt{1 + (\Lip \phi)^2}$.
  In particular, $u\big(x', y' + \phi(x')\big) \geq 1 - \eps$ whenever $y' \geq R(\eps) \sqrt{1 + (\Lip \phi)^2}$.
  This proves \eqref{eq:uniform}.
\end{proof}
The remainder of our argument for bistable nonlinearities hews close to~\cite{BCN97b}, which established uniqueness for monostable nonlinearities in uniformly Lipschitz epigraphs.
We frame a maximum principle deep in the interior of $\Omega$.
Given $h \geq 0$ and $v \colon \bar\Omega \to \R$, define the shifted domain $\Omega^h \coloneqq \Omega + h e_y$ and function $v^h \coloneqq v(\anon - h e_y)$.
We extend the latter by $0$ to $\R^d$.
\begin{proposition}
  \label{prop:MP-interior}
  Let $f$ be bistable~\ref{hyp:bistable}.
  There exists $H > 0$ such that the following holds for all $h \geq 0$.
  If $u$ and $v$ are positive bounded solutions of \eqref{eq:main} and $u \geq v^h$ on $\partial\Omega^H$, then $u \geq v^h$ in $\Omega^H$.
\end{proposition}
\begin{proof}
  Because $f'(1) < 1$ and $f'$ is continuous, there exists $s \in (0, 1)$ such that $\sup_{[s, 1]} f' \leq f'(1)/2 < 0$.
  Let $y' \coloneqq y - \phi(x')$.
  By Lemma~\ref{lem:limit}, there exists $H > 0$ such that $u \geq s$ where $y' \geq H$ for every positive bounded solution $u$ of \eqref{eq:main}.

  Given $h > 0$, suppose $u,v$ are two such solutions satisfying $u \geq v^h$ on $\partial \Omega^H$.
  Define $P \coloneqq \{v^h > u\} \cap \Omega^H$.
  There, the difference $w \coloneqq v^H - u$ satisfies the linear equation $-\Delta w = q w$ for $q \coloneqq \frac{f(v^H) - f(u)}{v^H - u}$.
  By the mean value theorem and the definitions of $H$ and $s$, $q \leq f'(1)/2 < 0$.
  It follows that
  \begin{equation*}
    \lambda(-\Delta - q, P) \geq \abs{f'(1)}/2 > 0.
  \end{equation*}
  So $-\Delta - q$ satisfies the maximum principle on $P$.
  Now, $w|_{\partial P} = 0$ because $v^h \leq u$ on $\partial \Omega^H$.
  It follows that $w = 0$ in $P$, so in fact $P$ is empty.
  That is, $v^h \leq u$ in $\Omega^H$, as desired.
\end{proof}
We are finally in a position to prove uniqueness in uniformly Lipschitz epigraphs that are bounded from below.
\begin{proof}[Proof of Theorem~\textup{\ref{thm:unique}}]
  Following the proof of~\cite[Proposition~2.2]{BG22}, we construct a subsolution $0 \lneqq \ubar{v} \leq 1$ of $-\Delta v = f(v)$ supported on a large ball.
  Since $\Omega$ contains arbitrarily large balls, we can place the support of $\ubar{v}$ within $\Omega$.
  Let $w$ solve the parabolic problem
  \begin{equation*}
    \begin{cases}
      \partial_t w = \Delta w + f(w) & \text{in }\Omega,\\
      w = 0 & \text{on } \partial \Omega,\\
      w(t = 0, \anon) = \ubar{v} & \text{in }\Omega.
    \end{cases}
  \end{equation*}
  Because $\ubar{v}$ is a subsolution, $w$ is increasing for all time.
  Moreover, the comparison principle implies that $w \leq 1$.
  It follows that $w$ has a long-time limit $w(\infty, \anon)$ that is positive, bounded, and satisfies \eqref{eq:main}.
  That is, we have existence.
  
  For uniqueness, let $u$ and $v$ be positive bounded solutions of \eqref{eq:main} and define
  \begin{equation*}
    \hs \coloneqq \sup \{h > 0 : v^h \not \leq u\}.
  \end{equation*}
  Recalling that $v^h = 0$ in $(\Omega^h)^\cc$, Proposition~\ref{prop:MP-interior} implies that $v^h \leq u$ for all $h \geq H$, so $\hs \leq H$.
  Suppose for the sake of contradiction that $\hs > 0$.
  By continuity, $v^\hs \leq u$.

  Fix $h_0 \in (0, \hs)$ and let $y' \coloneqq y - \phi(x')$.
  We claim that $\inf_{\{h_0 \leq y' \leq H\}} (u - v^\hs) > 0$, so toward a further contradiction, suppose instead that there exists a sequence $x_n = (x_n', y_n)$ with $y_n' \in [h_0, H]$ such that $(u - v^\hs)(x_n) \to 0$ as $n \to \infty$.
  Define
  \begin{equation*}
    \begin{gathered}
      u_n(x', y) \coloneqq u\big(x + x_n', y + \phi(x_n')\big), \enspace v_n(x', y) \coloneqq v\big(x + x_n', y + \phi(x_n')\big),\\
      \text{and} \quad \phi_n(x') \coloneqq \phi(x' + x_n') - \phi(x_n').
    \end{gathered}
  \end{equation*}
  Interior Schauder estimates yield a subsequential limit $(y_\infty', u_\infty, v_\infty, \phi_\infty)$ of the sequence $(y_n', u_n, v_n, \phi_n)$.
  We accordingly define $\Omega_\infty$, $\Omega_\infty^\hs,$ and $v_\infty^\hs \leq u_\infty$.
  We have $y_\infty' \in [h_0, H]$, so $(0, y_\infty') \in \Omega_\infty$.
  Also, Schauder estimates imply that $u_\infty$ and $v_\infty$ satisfy the analogue of \eqref{eq:main} in $\Omega_\infty$.
  Lemma~\ref{lem:limit} ensures that $u_\infty$ is nonzero, so by the strong maximum principle $u_\infty > 0$ in $\Omega_\infty$.  
  In particular, $u_\infty > 0 = v_\infty^\hs$ on $\partial \Omega_\infty^\hs$, so the strong maximum principle yields $u_\infty > v_\infty^\hs$ in $\Omega_\infty$.
  However, the definition of the sequence $(x_n)_n$ implies that $u_\infty(0, y_\infty') = v_\infty^\hs(0, y_\infty')$.
  This contradicts $(0, y_\infty') \in \Omega_\infty$, so indeed $\inf_{\{h_0 \leq y' \leq H\}} (u - v^\hs) > 0$.
  
  By the uniform continuity of $u$ and $v$, there exists $\eps > 0$ such that $u \geq v^h$ in $\{0 \leq y' \leq H\}$ for all $h \in (\hs - \eps, \hs]$.
  Because $u \geq v^h$ on $\{y' = H\} = \partial\Omega^H$, Proposition~\ref{prop:MP-interior} implies that $u \geq v^h$ everywhere.
  This contradicts the definition of $\hs$, so in fact $\hs = 0$.
  It follows that $v \leq u$, and by symmetry $u = v.$
\end{proof}
We can adapt these ideas to rule out certain solutions in general Lipschitz epigraphs, not necessarily bounded from below.
We recall that a function $\phi$ is \emph{coercive} if $\phi(x') \to \infty$ as $\abs{x'} \to \infty$.
\begin{theorem}
  \label{thm:no-ground}
  Let $\Omega$ be a uniformly Lipschitz epigraph and $f$ be bistable \ref{hyp:bistable} or field-type \ref{hyp:field}.
  Suppose \eqref{eq:main} admits a positive bounded solution $u$ that vanishes uniformly in $x'$ as $y \to -\infty$ and pointwise in $y$ as $\abs{x'} \to \infty$.
  (Both these limits may be empty.)
  Then $f$ is bistable, $\phi$ is coercive, and $u$ is unique.
\end{theorem}
\noindent
  If $\phi$ is coercive, then the vanishing conditions are vacuous and $\Omega$ is bounded from below, so Theorems~\ref{thm:unique} and \ref{thm:field-unique} apply.
  Thus Theorem~\ref{thm:no-ground} is essentially negative in nature: it forbids certain solutions in noncoercive epigraphs.
  In particular, it rules out ground states satisfying $u \to 0$ as $\abs{x} \to \infty$.
\begin{proof}
  Let $u$ be a positive bounded solution of \eqref{eq:main} that vanishes in the indicated manner.
  We adapt the proof of Proposition~\ref{prop:monotone} to show that $\partial_y u > 0$.
  Lemma~\ref{lem:limit} only uses the hypothesis that $\Omega$ is bounded from below to establish Proposition~\ref{prop:monotone}.
  It thus extends to this setting and states that $f$ is bistable and $u \to 1$ uniformly as $y' \coloneqq y - \phi(x') \to \infty$.
  This contradicts $u(x', y) \to 0$ as $\abs{x'} \to \infty$ unless $\phi$ is coercive, and then Theorem~\ref{thm:unique} yields uniqueness.
  So it suffices to show $\partial_y u > 0$.

  We recall notation from the proof of Proposition~\ref{prop:monotone}.
  Given $\mu \in \R$, define $\Sigma^\mu \coloneqq \{y < \mu\} \cap \Omega$, $v^\mu(x', y) \coloneqq u(x', 2 \mu - y)$, and $w^\mu \coloneqq v^\mu - u$, which satisfies \eqref{eq:diff} with $q^\mu \coloneqq \frac{f(v^\mu) - f(u)}{v^\mu - u}\tbf{1}_{v^\mu \neq u}$.
  Recall that there exists $s > 0$ such that $\sup_{[0, s]} f' \leq f'(0)/2 < 0$.
  Because $u \to 0$ uniformly as $y \to -\infty$, there exists $\mu_0 \in \R$ such that $u < s$ in $\Sigma^{\mu_0}$.

  Let $\mu \leq \mu_0$ and define $Q \coloneqq \{w^\mu < 0\}$.
  Then $v^\mu < u < s$ in $Q$, so $q^\mu \leq f'(0)/2$ by the mean value theorem and the definition of $s$.
  Hence $\lambda(-\Delta - q^\mu, Q) \geq f'(0)/2 > 0$, so \eqref{eq:diff} satisfies a maximum principle on $Q$.
  Because $w^\mu \geq 0$ on $\partial \Sigma^\mu$, we must have $w^\mu|_{\partial Q} = 0$.
  Hence $w^\mu = 0$ in $Q$.
  That is, $Q$ is in fact empty and $w^\mu \geq 0$ in $\Sigma^\mu$.

  Recall $\mus$ and $\Sigma_-$ from \eqref{eq:mu-crit} and \eqref{eq:compact}, so $\mus \geq \mu_0 > -\infty.$
  Suppose for the sake of contradiction that $\mus < \infty$.
  We claim that $\inf_{\Sigma_-} w^\mus > 0$.  
  To see this, we note that $\Sigma_- \subset \{u \geq s\}$.
  The definition of $\mu_0$ implies that $\Sigma_- \subset \R^{d-1} \times [\mu_0, \mus]$.
  Now $u(x', y) \to 0$ pointwise in $y$ as $\abs{x'} \to \infty$.
  Since $u$ is uniformly continuous, this limit in fact holds locally uniformly in $y$.
  In particular, $u \to 0$ as $\abs{x'} \to \infty$ uniformly in $y \in [\mu_0, \mus]$.
  It follows that $\Sigma_- \subset \Sigma^\mus$ is compact.
  The strong maximum principle yields $w^\mus > 0$ in $\Sigma^\mus$, so indeed $\inf_{\Sigma_-} w^\mus > 0$.
  
  From this point, the proof of Proposition~\ref{prop:monotone} proceeds unhindered, and we conclude that $\partial_y u > 0$.
  As observed at the beginning, this completes the proof.
\end{proof}

\section{Multiplicity}
\label{sec:multiple}

In this section, we instead consider domains $\Omega$ with aperture greater than $\pi$, meaning $\Omega$ contains a circular nonconvex cone.
We construct a large family of solutions of \eqref{eq:main} and thereby prove Theorems~\ref{thm:multiple} and \ref{thm:field-multiple}.
Our approach is inspired by the work of Malchiodi~\cite{Malchiodi}, who constructed solutions with ``Y-shape'' in the whole space $\R^d$.
We show that the presence of boundary can effectively replace one leg of the Y.
This boundary effect has an opposing sign, so an aperture greater than $\pi$ is required to accommodate the remaining two legs.

We work toward the following:
\begin{theorem}
  \label{thm:multiple-unified}
  Suppose $f$ is bistable \ref{hyp:bistable} or field-type \ref{hyp:field} and admits a nondegenerate ground state in the sense of \ref{hyp:ground}.
  Then if $\Omega$ is locally Lipschitz and has aperture greater than $\pi$, \eqref{eq:main} admits a $(d + 1)$-parameter family of positive bounded solutions.
\end{theorem}
From this result, Theorem~\ref{thm:multiple} is immediate and Theorem~\ref{thm:field-multiple} follows from Proposition~\ref{prop:nondegenerate}.

\subsection{Preliminaries}

Throughout, we assume that $f$ admits a nondegenerate ground state $U$ under satisfying~\ref{hyp:ground}.
For convenience, we normalize $f'(0) = -1$, which can be arranged by scaling $u$.
We then write $g(u) \coloneqq f(u) + u$ for the ``nonlinear part'' of $f$, viewed from the origin.

For the domain, we assume $\Omega \subsetneq \R^d$, for nonuniqueness is well-known in the free space.
Because $\partial\Omega$ is locally Lipschitz, there exists $\rho > 0$ such that some ball of radius $\rho$ is disjoint from $\Omega$.
We also assume $\Omega$ has aperture greater than $\pi$, so we can choose coordinates $x = (x', y) \in \R^{d-1} \times \R$ such that $\{y > -\al \abs{x'}\} \subset \Omega$ and $B_\rho(0, -\ell) \subset \Omega^\cc$ for some $\al,\ell > 0$.

\subsection{Delaunay solutions}

The solutions we construct resemble a bent chain of ground states spaced at wide but regular intervals.
At infinity, they converge to (rotations of) periodic ``Delaunay solutions'' defined on the fundamental domain $D_L \coloneqq \{\abs{y} < L/2\}$ and satisfying
\begin{equation}
  \label{eq:Delaunay}
  \begin{cases}
    -\Delta u = f(u) & \text{in } D_L,\\
    \partial_y u = 0 & \text{on } \partial D_L.
  \end{cases}
\end{equation}
Any positive solution of \eqref{eq:Delaunay} can be extended to an $L$-periodic solution of \eqref{eq:main} in $\R^d$.
In a minor abuse of notation, we use $u_L$ to refer to both.

When $f(u) = -u + u^p$ with $p \in (1, \frac{d+2}{d-2})$, Dancer used Crandall--Rabinowitz bifurcation theory to show that \eqref{eq:Delaunay} admits a positive bounded solution $u_L$ provided $L$ is sufficiently large~\cite{Dancer}.
The first and third authors~\cite{BW} have characterized $u_L$ as the minimizer of a certain energy when $L \gg 1$.
Here, we make use of estimates of Malchiodi~\cite{Malchiodi}, who used an implicit function theorem to construct $u_L$ converging exponentially quickly to the ground state $U$ as $L \to \infty$.
Although framed for $f(u) = -u + u^p$, the proof of~\cite[Proposition~3.1]{Malchiodi} extends to any nonlinearity with nondegenerate ground state.
We thus obtain:
\begin{proposition}[\cite{Malchiodi}]
  \label{prop:Delaunay}
  Let $f$ be bistable \ref{hyp:bistable} or field-type \ref{hyp:field} with nondegenerate ground state \ref{hyp:ground}.
  For $L$ sufficiently large, there exists a unique solution $u_L$ to \eqref{eq:Delaunay} with the following properties:
  \begin{enumerate}[label = \textup{(\roman*)}, itemsep = 2pt]
  \item $u_L$ is positive, bounded, radially symmetric in $x'$, and even in $y$.

  \item There exist $\xi,\sigma > 0$ such that for all $x \in D_L$,
    \begin{equation*}
      \abs{u_L(x) - U(x)} \lesssim \e^{-(1 + \xi)L/2} \e^{-\sigma \abs{x}}.
    \end{equation*}
  \end{enumerate}
\end{proposition}
When we extend $u_L$ to $\R^d$, we can interpret it as an exponentially-small perturbation of a chain of ground states strung along the lattice $\Lambda = \{(0, y) : y \in L \Z\}$.
Indeed, we can write $u_L(x) = \sum_{z \in \Lambda} U(x - z) + w_L(x)$ in $\R^d$ for a residue $w_L$ satisfying
\begin{equation}
  \label{eq:Delaunay-residue}
  \abs{w_L(x)}, \abs{\nab w_L(x)} \lesssim \e^{-(1 + \xi)L/2} \e^{-\sigma \dist(x, \Lambda)}.
\end{equation}
Following terminology arising in the study of singular limits, we refer to $U(x - z)$ as a ``spike.''

\subsection{Approximate solution}
We construct an approximate solution of \eqref{eq:main} resembling rotated copies of $u_L$ at infinity.
Recall that $\{y > -\al \abs{x'}\} \subset \Omega$ for some $\al > 0$.
Let $\Gamma \subset S^{d-1}$ denote the open set of directions inclined below the horizontal plane $\{y = 0\}$ by angle less than $\al/2$.
We distribute our spikes along rays in directions $\theta_\pm \in \Gamma$.
These rays tend downward but separate from $\partial\Omega$ at infinity, as in Figure~\ref{fig:multiple}.
Given $L_\pm \gg 1$, we space the spikes in direction $\theta_\pm$ by distance $L_\pm$, so the solution resembles $u_{L_\pm}$.
We shift this arrangement vertically a half-period $L_0 \approx L_\pm/2$ from $\partial\Omega$.
This is because the boundary behaves almost (but not exactly) like an ``oppositely charged'' spike at distance $2L_0$.
We choose the parameters $\theta_\pm$, $L_\pm$, and $L_0$ at the end of the construction.

In greater detail, let $z_0 \coloneqq (0, L_0) \in \R^2$ be our central point and take shifts $s_\pm \in \R^d$.
We place spikes at centers $\m{Z} = (z_k)_{k \in \Z}$ of the form
\begin{equation*}
  z_k = z_0 + s_{\pm} + \abs{k} \theta_\pm L_\pm + p_k
\end{equation*}
for $k \neq 0$, with signs matching $\sgn k$.
The sequence $(p_k)_{k \in \Z} \subset \R^d$ satisfies $p_0 = 0$ and $\abs{p_k} \lesssim \e^{-\sigma L_0 \abs{k}}$ for some $\sigma > 0$ and all $k \neq 0$.

Here $s_\pm$ represents the shift of the limiting Delaunay solution $u_{L_\pm}$ in direction $\theta_\pm$.
We adjust individual spikes by perturbations $p_k$ that decay exponentially away from $z_0$.
We write $\m{S} \coloneqq (s_-, s_+)$ and $\m{P} \coloneqq (p_k)_{k \in \Z}$, and we consider $\m{Z}$ to be parametrized by $\m{S}$ and $\m{P}$ (viewing $\theta_\pm$, $L_\pm$, and $L_0$ as fixed).
In contrast to \cite{Malchiodi}, we do not restrict $\theta_\pm$, $s_\pm$, and $p_k$ to a plane.
This full-dimensional freedom allows us to compensate for potential asymmetries in $\partial\Omega$.

To incorporate the Delaunay residue $w_L$ from \eqref{eq:Delaunay-residue}, let $\m{R}_\pm$ denote rotation by $\theta_\pm$ and define
\begin{equation*}
  w_\pm(x) \coloneqq (\m{R}_\pm w_{L_\pm})(x - z_0 - s_\pm).
\end{equation*}
This is an apt correction to the spike chain far from the center, but it is less effective near $z_0$.
We therefore cut it off.
Let $\psi \in \m{C}^\infty(\R)$ satisfy $\psi|_{(-\infty, 0]} \equiv 0$ and $\psi|_{[1, \infty)} \equiv 1$ and define
\begin{equation*}
  \psi_\pm(x) \coloneqq \psi\big((x - z_0) \cdot \theta_\pm - L_\pm/2\big).
\end{equation*}
We also introduce an angular cutoff to ensure we satisfy the Dirichlet condition.
Let $\chi \in \m{C}^\infty(\R^d)$ satisfy $\chi \equiv 0$ where $y \leq -\al \abs{x'}$ and $\chi \equiv 1$ where $y \geq -\tfrac{3}{4}\al \abs{x'} + 1$.

These cutoffs are designed to treat all but the central spike.
For $U_0 \coloneqq U(\anon - z_0)$, we must take greater care.
Let $\bar{U}_0$ denote the ``Dirichlet projection'' of $U_0$, namely the unique decaying solution of the linear elliptic problem
\begin{equation}
  \label{eq:projection}
  \begin{cases}
    -\Delta \bar{U}_0 + \bar{U}_0 = g(U_0) & \text{in } \Omega,\\
    \bar{U}_0 = 0 & \text{on } \partial\Omega,
  \end{cases}
\end{equation}
recalling $g(u) \coloneqq f(u) + u$.
With this notation, we can finally define our approximate solution
\begin{equation}
  \label{eq:approx}
  \SP{u}(x) \coloneqq \bar{U}_0(x) + \sum_{k \neq 0} \chi(x) U(x - z_k) + \chi(x)\big[\psi_-(x) w_-(x) + \psi_+(x)w_+(x)\big].
\end{equation}

\subsection{Controlling the projection}
Our analysis of the approximate solution $\SP{u}$ requires a thorough understanding of the projection $\bar{U}_0$ from \eqref{eq:projection}, which provides the central spike in our construction.
We broadly follow the approach of~\cite{NW}, with some quantitative improvements.

We expect $\bar{U}_0$ to resemble the unmodified spike $U_0$, so define the residue $\varphi_0 \coloneqq U_0 - \bar{U}_0$.
This satisfies
\begin{equation*}
  \begin{cases}
    -\Delta \varphi_0 + \varphi_0 = 0 & \text{in } \Omega,\\
    \varphi_0 = U_0 & \text{on } \partial\Omega,
  \end{cases}
\end{equation*}
and by the maximum principle, $\varphi_0 > 0$.
To control $\varphi_0$, we note that a standard ODE argument (\cite[Theorem~2]{GNN81}) yields $A > 0$ such that
\begin{equation}
  \label{eq:ground-decay}
  U(r),-U'(r) = A r^{-\frac{d-1}{2}} \e^{-r} \Big[1 + \m{O}\Big(\frac{1}{r \vee 1}\Big)\Big].
\end{equation}
We also make use of the Green function $K$ satisfying $-\Delta K + K = \delta_0$.
It obeys a nearly identical bound (\cite[Appendix~C]{GNN81}) for some $\kappa > 0$:
\begin{equation}
  \label{eq:Green}
  K(r), -K'(r), K''(r) = \kappa r^{-\frac{d-1}{2}} \e^{-r} [1 + \m{O}(r^{-1})] \quad \text{where } r \geq 1.
\end{equation}
With these estimates, we show:
\begin{lemma}
  \label{lem:projection}
  There exists a constant $C > 0$ such that for all $L_0 \geq C$ and $x \in \Omega$,
  \begin{equation}
    \label{eq:varphi-global}
    C^{-1} L_0^{-\frac{d-1}{2}} (\abs{x}+1)^{-\frac{d-1}{2}} \e^{-L_0} \e^{-\abs{x}} \leq \varphi_0(x) \leq C \min\{K(\abs{x - z_0}), K(\abs{x + z_0})\}.
  \end{equation}
  In particular, $\abss{\bar{U}_0} \leq C U_0$.
\end{lemma}
\begin{proof}
  For the upper bound, we observe that \eqref{eq:ground-decay} and \eqref{eq:Green} imply that $U_0 \lesssim K(x - z_0)$ on $\partial\Omega$.
  Since $K$ is a supersolution of $-\Delta \varphi + \varphi = 0$, the maximum principle implies that $\varphi_0 \lesssim K(x - z_0)$.
  Moreover, because $\varphi_0$ is bounded, we have $0 < \varphi_0 \lesssim U_0$, which implies $\abss{\bar{U}_0} \lesssim U_0$.
  On the other hand, because $y < 0$ on $\partial\Omega$ and $U$ is radially decreasing, we also have
  \begin{equation*}
    U_0(x - z_0) = U(\abs{x - z_0}) \leq U(\abs{x + z_0}) \lesssim K(x + z_0) \quad \text{on } \partial\Omega.
  \end{equation*}
  By the same reasoning, $\varphi_0 \lesssim K(x + z_0)$.

  For the lower bound, we have chosen coordinates so that $B_\rho(q_0) \subset \Omega^\cc$ for $q_0 = (0,-\ell)$ and some $\ell > 0$ independent of $L_0$.
  By the triangle inequality,
  \begin{equation*}
    \abs{x - z_0} \leq \abs{x - q_0} + \abs{z_0 - q_0} = \abs{x - q_0} + \ell + L_0.
  \end{equation*}
  Also, $\abs{x - q_0} \geq \rho$ implies that $\abs{x - q_0} + \ell + L_0 \lesssim L_0 \abs{x - q_0}$.
  Hence on $\partial\Omega$, \eqref{eq:ground-decay} and \eqref{eq:Green} yield
  \begin{align*}
    U_0(x) \geq U(\abs{x - q_0} + \ell + L_0) &\gtrsim (\abs{x - q_0}+L_0)^{-\frac{d-1}{2}} \e^{-L_0} \e^{-\abs{x - q_0}}\\
                                              &\gtrsim L_0^{-\frac{d-1}{2}} \e^{-L_0} K(x - q_0).
  \end{align*}
  Because $q_0 \not \in \Omega$, $K(x - q_0)$ solves $-\Delta \varphi + \varphi = 0$ in $\Omega$.
  The desired lower bound follows from the maximum principle and \eqref{eq:Green}.
\end{proof}
We are particularly interested in the behavior of $\varphi_0$ near the central spike at $z_0$.
\begin{lemma}
  \label{lem:center}
  There exist $C > 1$ such that for all $L_0 \geq C$,
  \begin{gather}
    C^{-1} L_0^{-(d-1)} \e^{-2L_0} \leq \varphi_0(z_0) \leq C L_0^{-(d-1)} \e^{-2L_0},\label{eq:center}\\
    C^{-1} \varphi_0(z_0) \leq - \partial_{L_0}[\varphi_0(z_0)] \leq C \varphi_0(z_0).\label{eq:derivative}
  \end{gather}
  Also, there exists $\beta > 0$ such that for all $\abs{x'} \in \R^{d-1}$,
  \begin{equation}
    \label{eq:horizontal}
    \varphi_0(x',0) \leq C L^{-\frac{d-1}{2}} \e^{-L_0} \e^{-\beta\abs{x'}}.
  \end{equation}
\end{lemma}
\noindent
We use the final estimate as input in the next lemma.

By \eqref{eq:center}, $\varphi_0(z_0)$ agrees with $U(2L_0)$ at exponential order, as one may expect from the method of images.
However, these quantities differ by a polynomial factor.
As a result, the boundary behaves slightly differently from a negative spike placed at $-z_0$.
\begin{proof}
  Where $y < 0$, we have $\abs{x - z_0} \geq L_0 - y \geq L_0$.
  By \eqref{eq:ground-decay} and \eqref{eq:varphi-global}, we find
  \begin{equation}
    \label{eq:cone-boundary}
    \varphi_0(x) \lesssim L_0^{-\frac{d-1}{2}} \e^{-L_0} \e^{-\abs{y}}.
  \end{equation}
  Define the cone $J \coloneqq \{y < -\al\abs{x'}\}$ and let $\psi$ solve
  \begin{equation*}
    -\Delta \psi + \psi = \tbf{1}_J \e^{-\abs{y}}.
  \end{equation*}
  Then we can write
  \begin{equation*}
    \psi(x) = \int_J K(x - z) \e^{-\abs{z_d}} \d z.
  \end{equation*}
  For all $x \in \partial J = \{y = -\al \abs{x'}\}$, we have $\abs{B_1(x) \cap J} \gtrsim 1$.
  Because $K$ is decreasing in $r$, we obtain
  \begin{equation*}
    \psi(x) \geq \e^{-\abs{y}-1} \int_{B_1(x) \cap J} K(x - z) \d z \geq \e^{-\abs{y}-1} K(1) \abs{B_1(x) \cap J} \gtrsim \e^{-\abs{y}}.
  \end{equation*}
  By hypothesis, $\partial\Omega \subset J$.
  Since $\psi(x) \gtrsim \e^{-\abs{y}}$ on $\partial J$, \eqref{eq:cone-boundary} and the maximum principle yield
  \begin{equation}
    \label{eq:varphi-psi}
    \varphi_0 \lesssim L_0^{-\frac{d-1}{2}} \e^{-L_0} \psi \quad \text{in } J^\cc.
  \end{equation}
  Now for each $z \in J$, $\abs{z - z_0} \geq L_0 + \abs{z_d} \geq L_0$.
  Hence by \eqref{eq:Green},
  \begin{equation*}
    \psi(z_0) = \int_J K(z_0 - z) \e^{-\abs{z_d}} \d z \lesssim L_0^{-\frac{d-1}{2}} \e^{-L_0} \int_J \e^{-2\abs{z_d}} \d z \lesssim L_0^{-\frac{d-1}{2}} \e^{-L_0}.
  \end{equation*}
  Here we use the conical nature of $J$ to ensure that the final integral is finite.
  Then the upper bound  in \eqref{eq:center} follows from \eqref{eq:varphi-psi}, and the lower from \eqref{eq:varphi-global}.

  Now consider \eqref{eq:horizontal}.
  By the Harnack inequality, it suffices to consider $x' \in \R^{d-1}$ such that $r \coloneqq \abs{x'} \geq 1$.
  Then we use \eqref{eq:Green} to write
  \begin{equation*}
    \psi(x',0) \lesssim \int_J \e^{-\abs{(x',0) - z} - \abs{z_d}} \d z
  \end{equation*}
  A routine calculation shows that
  \begin{equation*}
    \min_{z \in J} (\abs{(x',0) - z} + \abs{z_d}) = \frac{2 \al}{1 + \al^2} r.
  \end{equation*}
  Fix $\beta \in \big(0, \frac{2\al}{1 + \al^2}\big)$.
  Then Laplace's method yields $\psi(x', 0) \lesssim \e^{-\beta r}$, and \eqref{eq:horizontal} follows from \eqref{eq:varphi-psi}.

  For the derivative estimate \eqref{eq:derivative}, we observe that $\varphi_0$ itself depends on $L_0$ through its boundary condition.
  We can write
  \begin{equation}
    \label{eq:chain}
    -\partial_{L_0}[\varphi_0(z_0)] = -\partial_y \varphi_0(z_0) - (\partial_{L_0}\varphi_0)(z_0).
  \end{equation}
  Using a kernel representation for $\varphi_0 > 0$ on $\{y > 0\}$, one can readily check that $\partial_y \varphi_0 < 0$ there, and Schauder and Harnack estimates yield $0 < -\partial_y \varphi_0(z_0) \lesssim \varphi_0(z_0)$.
  For the second term, we let $w \coloneqq -\partial_{L_0} \varphi_0$, which satisfies
  \begin{equation*}
    \begin{cases}
      -\Delta w + w = 0 & \text{in } \Omega,\\
      w = -\partial_y U_0 & \text{on } \partial\Omega.
    \end{cases}
  \end{equation*}
  We can write
  \begin{equation*}
    -\partial_y U_0(x) = \partial_y[U(x - z_0)] = -\frac{L_0 - y}{\abs{x - z_0}} U'(x - z_0).
  \end{equation*}
  Because $y \leq -\al \abs{x'}$ on $\partial \Omega$, we have $\frac{L_0 - y}{\abs{x - z_0}} \asymp 1$.
  Moreover, \eqref{eq:ground-decay} implies that $-U' \asymp U$ outside the unit ball.
  So $w|_{\partial \Omega} \asymp \varphi_0|_{\partial \Omega}$.
  It follows from the maximum principle that $w \asymp \varphi_0$.
  Combining these observations in \eqref{eq:chain}, we obtain \eqref{eq:derivative}.
\end{proof}
We next determine the exponential character of $\varphi_0$ near $z_0$:
\begin{lemma}
  \label{lem:exponential}
  There exists $C > 1$ such that for all $L_0 \geq C$ and $\abs{x} \leq L_0^{1/3}$,
  \begin{equation}
    \label{eq:uniform-ratio}
    \Big|\frac{\varphi_0(x + z_0)}{\varphi_0(z_0)} \e^y - 1\,\Big| \leq C L_0^{-1/3} \And \abs{\nab_{x'} \varphi_0(x + z_0)} \leq C L_0^{-2/3} \e^{-y} \varphi_0(z_0).
  \end{equation}
\end{lemma}
\begin{proof}
  It is well known that we can write
  \begin{equation*}
    \varphi_0(x + z_0) = -2\int_{\R^{d-1}} \partial_y K(x' - z', y + L_0) \varphi_0(z', 0) \d z'.
  \end{equation*}
  Fix $\abs{x} \leq L_0^{1/3}$ and let $D$ denote the ball of radius $L_0^{1/3}$ in $\R^{d-1}$.
  We claim the contributions to this integral from $D^\cc$ are negligible.
  Indeed by \eqref{eq:Green}, we have $\abss{\partial_y K(x' - z', y + L_0)} \lesssim \e^{-y} \e^{-L_0}$.
  Using \eqref{eq:center} and \eqref{eq:horizontal} from Lemma~\ref{lem:center}, we find
  \begin{equation*}
    \int_{D^\cc} \abs{\partial_y K(x' - z', y + L_0) \varphi_0(z', 0)} \d z' \lesssim \e^{-y-2L_0} \int_{D^\cc} \e^{-\beta \abs{x'}} \d z' \lesssim \e^{-y - c L_0^{1/3}} \varphi_0(z_0)
  \end{equation*}
  for $c \in (0, \beta)$.
  As a consequence,
  \begin{equation}
    \label{eq:cutoff}
    \varphi_0(x + z_0) = -2\int_D \partial_y K(x' - z', y + L_0) \varphi_0(z', 0) \d z' + \m{O}\big[\e^{-c L_0^{1/3}} \e^{-y} \varphi_0(z_0)\big].
  \end{equation}

  Now let $d_1 \coloneqq \abs{(x' - z', y + L_0)}$ and $d_0 \coloneqq \abs{(z',L_0)}$.
  Then
  \begin{equation*}
    d_1 = \sqrt{L_0^2 + 2 y L_0 + \m{O}\big(L_0^{2/3}\big)} = L_0 \sqrt{1 + 2y L_0^{-1} + \m{O}\big(L_0^{-4/3}\big)} = L_0 + y + \m{O}\big(L_0^{-1/3}\big)
  \end{equation*}
  and similarly $d_0 = L_0 + \m{O}\big(L_0^{-1/3}\big)$.
  Thus \eqref{eq:Green} implies that
  \begin{equation*}
    \frac{\partial_y K(x' - z', y + L_0)}{\partial_y K(- z', L_0)} = \e^{-y} \big[1 + \m{O}(L_0^{-1/3})\big].
  \end{equation*}
  Using this ratio in \eqref{eq:cutoff}, we obtain the first part of \eqref{eq:uniform-ratio}.

  For the second, we write
  \begin{equation*}
    \nab_{x'}\varphi_0(x + z_0) = -2\int_{\R^{d-1}} \partial_y \nab_{x'} K(x' - z', y + L_0) \varphi_0(z', 0) \d z'.
  \end{equation*}
  Let $p \coloneqq (x' - z', y + L_0)$.
  Given $1 \leq i \leq d-1$, radial symmetry yields
  \begin{equation*}
    \abs{\partial_{x_i} K(p)} = \frac{\abs{x_i' - z_i'}}{\abs{p}} \abs{K'(\abs{p})} \lesssim L_0^{-2/3} K(p).
  \end{equation*}
  One can similarly check that $\abss{\partial_y \nab_{x'}K(p)} \lesssim L_0^{-2/3} K(p)$.
  Then the second part of \eqref{eq:uniform-ratio} follows from the estimates above. 
\end{proof}
This exponential limit will play an important role in determining the angles $\theta_\pm$ and the lengths $L_\pm$ and $L_0$.
Our analysis also involves coefficients consisting of the following integrals:
\begin{lemma}
  \label{lem:positive}
  We have
  \begin{equation*}
    D \coloneqq \int_{\R^d} g'(U) \partial_y U \e^{-y} = \int_{\R^d} g(U) \e^{-y} > 0 \And E \coloneqq \int_{\R^d}g(U) (\partial_y U)^2 > 0.
  \end{equation*}
\end{lemma}
\begin{proof}
  Because $f \in \m{C}_{\text{loc}}^{1,\gamma}$ and $f'(0) = -1$, $g'(0) = 0$ and
  \begin{equation}
    \label{eq:g-deriv}
    \abs{g'(s)} \lesssim s^\gamma \quad \text{for } s \in [0, \sup U].
  \end{equation}
  Hence by \eqref{eq:ground-decay}, $\abss{g'(U) \partial_y U} \lesssim \e^{-(1 + \gamma)r}$.
  The first integral defining $D$ is thus absolutely convergent, so the second follows from integration by parts.
  
  We symmetrize $\e^{-y}$ to obtain an equivalent integral.
  Let $V_*$ denote the unique radially-symmetric positive solution of $-\Delta V_* + V_* = 0$ in $\R^d$ such that $V_*(0) = 1$.
  This is given by
  \begin{equation*}
    V_*(x) = \frac{1}{\abss{S^{d-1}}} \int_{S^{d-1}} \e^{\theta \cdot x} \d \theta.
  \end{equation*}
  Hence Laplace's method yields a constant $B > 0$ such that
  \begin{equation}
    \label{eq:V-asymp}
    V_*(r),V_*'(r) \sim B r^{-\frac{d-1}{2}} \e^r \quad \text{as } r \to \infty.
  \end{equation}
  Returning to $D$, symmetry implies
  \begin{equation*}
    D = \int_{\R^d} g(U) \e^{-y} = \int_{\R^d} g(U) V_*.
  \end{equation*}
  Since $g(U) = -\Delta U + U$, we can now write $D = \lim_{R \to \infty} D_R$ for
  \begin{equation*}
    D_R \coloneqq \lim_{R \to \infty} \int_{B_R} (-\Delta U + U) V_* = \int_{B_R} (U\Delta V_* - V_* \Delta U).
  \end{equation*}
  Integrating by parts, we find
  \begin{equation*}
    D_R = \int_{\partial B_R} (U V_*' - V_* U') = \abss{S^{d-1}} R^{d-1}[U(R)V_*'(R) - V_*(R)U'(R)].
  \end{equation*}
  Combining \eqref{eq:ground-decay} and \eqref{eq:V-asymp}, we see that
  \begin{equation*}
    D_R \to 2AB \abss{S^{d-1}} > 0 \quad \text{as } R \to \infty.
  \end{equation*}
  That is, $D = 2AB \abss{S^{d-1}} > 0$.

  We now turn to $E$.
  Integrating by parts, $E = - \int_{\R^d} g(U) \partial_y^2U$.
  Again using $g(U) = -\Delta U + U$, we repeatedly integrate by parts to obtain
  \begin{equation*}
    E = \int_{\R^d} (-\Delta U + U)(-\partial_y^2 U) = \int_{\R^d} \big(\abss{\nab \partial_y U}^2 + \abss{\partial_y U}^2\big) > 0.
    \qedhere
  \end{equation*}
\end{proof}

\subsection{Error and corrector}
We now show that $\SP{u}$ from \eqref{eq:approx} approximately satisfies \eqref{eq:main}.
\begin{lemma}
  \label{lem:approx}
  There exist $C,\xi,\sigma > 0$ such that for all $L_0 \geq C$ and $x \in \Omega$,
  \begin{equation*}
    \abs{\Delta \SP{u} + f(\SP{u})} \leq C \e^{-(1 + \xi)L_0} \e^{-\sigma \abs{x}}.
  \end{equation*}
\end{lemma}
\begin{proof}
  This error estimate is very similar to Lemma~4.1 in~\cite{Malchiodi}.
  For this reason, we only discuss the novel aspects of our setting: the projection $\bar{U}_0$ and the cutoff $\chi$.

  We first consider the residue due to $\bar{U}_0$ alone.
  Recalling that $\varphi_0 = U_0 - \bar{U_0} > 0$, \eqref{eq:g-deriv} and the mean value theorem yield
  \begin{equation*}
    \abss{\Delta \bar U_0 + f(\bar{U}_0)} = \abss{g(U_0) - g(\bar{U}_0)} \lesssim \sup_{[0,U_0]}\abs{g'} \varphi_0 \lesssim U_0^\gamma \varphi_0.
  \end{equation*}
  Writing $U_0^\gamma \lesssim \e^{-\gamma\abs{x - z_0}}$, Lemma~\ref{lem:projection} readily implies that
  \begin{equation*}
    \abss{\Delta \bar U_0 + f(\bar{U}_0)} \lesssim \e^{-(1 + \xi)L_0} \e^{-\sigma \abs{x}}
  \end{equation*}
  for suitable $\xi, \sigma > 0$.
  Whenever $\bar{U}_0$ interacts with other terms in $\SP{u}$, we use Lemma~\ref{lem:projection} to bound it by $CU_0$.

  Next, the cutoff $\chi$ introduces error in the region $\m{A} \coloneqq \{-\al \abs{x'} \leq y \leq -\tfrac{3\al}{4}\abs{x'} + 1\}$ where it transitions from $1$ to $0$.
  These errors are accompanied by factors involving $Q$ or $\nab Q$ for
  \begin{equation}
    \label{eq:Q}
    Q(x) \coloneqq  \sum_{k \neq 0} U(x - z_k) + \big[\psi_-(x) w_-(x) + \psi_+(x)w_+(x)\big].
  \end{equation}
  Recall that the spike centers $\m{Z}$ are strung along rays inclined at angles $\theta_\pm$.
  Because $\theta_\pm \in \Gamma$, these rays are angularly separated by at least $\al/4$ from the cone $y \leq -\tfrac{3\al}{4} \abs{x'}$.
  Geometric considerations thus provide $\delta(\al) > 0$ such that
  \begin{equation*}
    \op{dist}(x,\m{Z} \setminus \{0\}) \geq (1 + \delta)L_0 + \delta \abs{x} \quad \text{for all } x \in \m{A}.
  \end{equation*}
  Combining \eqref{eq:Delaunay-residue} and \eqref{eq:ground-decay}, we see that
  \begin{equation}
    \label{eq:Q-small}
    Q(x),\abs{\nab Q} \lesssim \e^{-(1 + \xi)L_0} \e^{-\sigma \abs{x}} \quad \text{in } \m{A}
  \end{equation}
  for sufficiently small $\xi, \sigma$.
  Hence residue arising from the cutoff $\chi$ is suitably small.

  The remaining contributions to the residue $\Delta \SP{u} + f(\SP{u})$ are controlled as in the proof of \cite[Lemma~4.1]{Malchiodi}.
\end{proof}
With this error estimate, the Lyapunov--Schmidt reduction in~\cite{Malchiodi} produces a corrector $\SP{w}$ to bring us closer to a true solution of \eqref{eq:main}.
Letting $U_k \coloneqq U(\anon - z_k)$ and recalling the constants $D,E > 0$ from Lemma~\ref{lem:positive}, the following is essentially a restatement of Proposition~4.3 from \cite{Malchiodi}.
\begin{proposition}[\cite{Malchiodi}]
  \label{prop:corrector}
  There exist $C,\xi,\sigma > 0$ such that for all $L_0 \geq C$, there exist $\SP{w} \colon \Omega \to \R$ and $(\eta_k)_{k \in \Z} \subset \R^d$ satisfying:
  \begin{enumerate}[label = \textup{(\roman*)}, itemsep = 2pt]
  \item
    $\displaystyle -\Delta (\SP{u} + \SP{w}) - f(\SP{u} + \SP{w}) = \frac{D}{E}\sum_{k \in \Z} g'(U_k) \eta_k \cdot \nab U_k.$
    \hfill 
    \puteqnum \label{eq:force-residue}

  \item
    For all $k \in \Z$, $\int_\Omega \SP{w} g'(U_k) \nab U_k = 0$.

  \item
    \label{item:small-corrector}
    For all $x \in \Omega$, $\abss{\SP{w}(x)} \leq C \e^{-(1 + \xi) L_0} \e^{-\sigma \abs{x}}.$
  \end{enumerate}
\end{proposition}
That is, there exists a small corrector $\SP{w}$ that eliminates all error orthogonal to the span of $\big(\nab g(U_k)\big)_k$ in $L^2(\Omega)$.
As a consequence, to construct a true solution of \eqref{eq:main}, it suffices to choose parameters so that $\eta_k = 0$ for all $k \in \Z$.

We note in passing that $u \coloneqq \SP{u} + \SP{w}$ satisfies
\begin{equation}
  \label{eq:Delaunay-limit}
  u(x) \to (\m{R}_\pm u_{L_\pm})(x - z_0 - s_\pm)
\end{equation}
uniformly as $x \cdot \theta_\pm \to \infty$.
Here $\m{R}_\pm$ denotes rotation by $\theta_\pm$ and $u_{L_\pm}$ is the Delaunay solution of period $L_\pm$ from Proposition~\ref{prop:Delaunay}.
This is a consequence of \eqref{eq:approx}, Proposition~\ref{prop:corrector}\ref{item:small-corrector}, and the fact that $\abs{z_k} \to 0$ as $\abs{k} \to \infty$.
We use \eqref{eq:Delaunay-limit} to prove the instability of $u$ in Section~\ref{sec:stable}.

\subsection{Force balance}
Formally, we can interpret the spikes comprising $\SP{u}$ as point masses that exert forces on one another through the variation of the Euler--Lagrange functional associated to \eqref{eq:main}.
The vector $\eta_k$ expresses the balance of forces acting on spike $U_k$.
We wish to tune parameters to achieve an equilibrium in which all forces vanish.
By \eqref{eq:force-residue}, this corresponds to a solution of \eqref{eq:main}.

The computation of $\eta_k$ is very similar to the treatment of the coefficients $\al_{X,Y}^I$ in Section~5 of \cite{Malchiodi}.
We introduce the shorthand $F(r) \coloneqq r^{-(d-1)/2} \e^{-r}$ and let $\almost$ denote an expression of the form $1 + \smallO(1)$ as $L_0 \to \infty$ that may change from instance to instance.
Recalling $A > 0$ from \eqref{eq:ground-decay}, we have a variant of \cite[Lemma~5.3]{Malchiodi}:
\begin{lemma}
  \label{lem:forces}
  There exist $\xi, \sigma > 0$ such that for all $L_0$ sufficiently large,
  \begin{align}
    \eta_0 &= \almost \varphi_0(z_0) e_y + \sum_{k = \pm 1} \almost A F(\abs{z_k - z_0}) \frac{z_k - z_0}{\abs{z_k - z_0}} + \m{O}\big(L_0^{-2/3}\varphi_0(z_0)\big),\label{eq:eta0}\\
    \eta_{\pm 1} &= \almost A F(\abs{z_0 - z_{\pm 1}}) \frac{z_0 - z_{\pm 1}}{\abs{z_0 - z_{\pm 1}}} + \almost A F(\abs{z_{\pm 2} - z_{\pm 1}}) \frac{z_{\pm 2} - z_{\pm 1}}{\abs{z_{\pm 2} - z_{\pm 1}}} + \m{O}\big(\e^{-(1 + \xi)2L_0}\big),\nonumber\\[4pt]
    \eta_k &= \m{O}\big(\e^{-(1 + \xi)2L_0} \e^{-\sigma L_0 \abs{k}}\big) \quad \text{for } \abs{k} \geq 2.\nonumber
  \end{align}
\end{lemma}
From these expressions, we see that the boundary repels spikes, but they attract one another.
This fact necessitates the wide aperture of $\Omega$: if $\Omega$ lay within a convex cone, spike chains in the interior would pull the central point away from the boundary, joining rather than canceling the repulsion from the boundary.
There would be no equilibria and no spike solutions, in agreement with the uniqueness results from Section~\ref{sec:unique}.

We note that the boundary repulsion on all but the central spike is exponentially smaller than the inter-spike attraction, so we group it with the error.
We can similarly ignore the attraction between non-adjacent spikes.
Moreover, the smallness of the perturbations $p_k$ ensures that the attractive forces on all spikes $k$ with $\abs{k} \geq 2$ very nearly cancel.
For this reason, the forces $\eta_{-1},\eta_0,$ and $\eta_1$ dominate.
\begin{proof}
  Following \cite{Malchiodi}, the leading part of $\eta_k$ can be computed by multiplying \eqref{eq:force-residue} by $\nab U_k$ and integrating over $\Omega$.
  For most details, we direct the reader to the proof of \cite[Lemma~5.3]{Malchiodi}.
  Here, we just treat the differences introduced by the presence of boundary.

  For $\eta_0$, the new contribution comes from boundary repulsion.
  When we integrate the left side of \eqref{eq:force-residue} against $\nab U_0$, the central spike contributes at leading order due to the projection.
  Tracking these terms, we have:
  \begin{equation*}
    I \coloneqq \int_\Omega [-\Delta \bar U_0 + \bar{U}_0 - g(\bar{U}_0)] \nab U_0 = \int_\Omega [g(U_0) - g(\bar{U}_0)] \nab U_0.
  \end{equation*}
  Using $g \in \m{C}^{1,\gamma}$, the mean value theorem yields $g(U_0) - g(\bar{U}_0) = \varphi_0 g'(U_0) + \m{O}(\varphi_0^{1 + \gamma})$.
  We can then readily check that
  \begin{equation*}
    I = \int_\Omega \varphi_0 g'(U_0) \nab U_0 + \m{O}\big(\e^{-(1 + \xi)2L_0}\big).
  \end{equation*}
  Using Lemma~\ref{lem:exponential} where $\abs{x - z_0} \leq L_0^{1/3}$ and Lemma~\ref{lem:projection} elsewhere, we find
  \begin{equation*}
    I = \almost \varphi_0(z_0) \int_{\R^d} \e^{-y} g'(U) \partial_y U e_y + \m{O}\big(L_0^{-2/3}\varphi_0(z_0)\big).
  \end{equation*}
  Then Lemma~\ref{lem:positive} yields
  \begin{equation*}
    I = \almost D \varphi_0(z_0) e_y + \m{O}\big(L_0^{-2/3}\varphi_0(z_0)\big).
  \end{equation*}
  Finally, adapting Lemma~A.6 in~\cite{Malchiodi} to our notation, this integral contributes $I/D$ to $\eta_0$.
  This explains the first term in \eqref{eq:eta0}.

  The remaining terms agree with \cite[Lemma~5.3]{Malchiodi}, so it only remains to show that the error from the cutoff $\chi$ is small.
  As in the proof of Lemma~\ref{lem:approx}, this cutoff error is supported in $\m{A} \coloneqq \{-\al \abs{x'} \leq y \leq -\tfrac{3\al}{4}\abs{x'} + 1\}$ and is accompanied by factors involving $Q$ from \eqref{eq:Q}.
  For example, we obtain contributions of order $\int_{\m{A}} Q \nab U_k$ for $k \neq 0$.
  Recalling \eqref{eq:Q-small} and the geometry of $\m{A}$ relative to $\m{Z}$, we conclude that this term is $\m{O}\big(\e^{-(1 + \xi)2L_0} \e^{-\sigma L_0 \abs{k}}\big)$.
  Similar considerations hold for all the residues introduced by $\chi$, so this angular cutoff does not alter the final result.
\end{proof}
From this stage, the methods of Sections~6 and 7 in \cite{Malchiodi} apply and yield an equilibrium solution with $\eta_k = 0$ for all $k$.
Briefly, one shows that the map $(\m{S},\m{P}) \to (\eta_k)_k$ is invertible about $\m{P} \equiv 0$ in an exponentially weighted space.
Then a contraction mapping argument shows that for each fixed $(\theta_\pm, L_\pm, L_0)$, we can choose $(\m{S},\m{P})$ so that $\eta_k = 0$ for all $k \neq 0$.
Thus it remains to satisfy the central force balance \eqref{eq:eta0}.

By Lemma~\ref{lem:center}, the boundary repulsion and inter-spike attraction have strengths
\begin{equation*}
  \varphi_0(z_0) \asymp L_0^{-(d-1)} \e^{-2L_0} \And F(z_{k + 1} - z_{k}) \asymp L_\pm^{-\frac{d-1}{2}}\e^{-L_\pm}.
\end{equation*}
We choose $L_\pm = 2L_0 + \frac{d-1}{2} \log L_0 + \m{O}(1)$ so that these are of the same order.
The logarithmic correction arises from the slightly different character of the force from the boundary, as discussed after Lemma~\ref{lem:center}.

In total, we have $2d + 1$ scalar parameters, from the two angles in $S^{d-1}$ and the three lengths, and we need to satisfy $d$ scalar equations for $\eta_0 = 0$ in $\R^d$.
A degree-theory argument yields a $(d + 1)$-parameter family of solutions.
For example, we are free to fix $\theta_-$, $L_-$, and $L_0$ and vary $\theta_+$ and $L_+$ to obtain a solution.
If we wish to vary $L_0$ (among other parameters), we use \eqref{eq:derivative}.
We note that due to the structure of \eqref{eq:eta0}, the $x'$ components of $\theta_\pm$ must be nearly anti-aligned.
This completes the proof of Theorem~\ref{thm:multiple-unified}.

\section{Stability}
\label{sec:stable}

We now consider the stability of solutions.
We show that epigraphs admit at most one strictly stable solution, while the multitude of solutions constructed in Theorems~\ref{thm:multiple} and \ref{thm:field-multiple} are strictly unstable.

\subsection{Stable solutions in epigraphs}
\label{sec:stable-epigraphs}

We assume a modest amount of additional regularity for our stability result.
\begin{theorem}
  \label{thm:stable-expanded}
  Let $\Omega$ be a uniformly $\m{C}^{1,\al}$ epigraph for some $\al \in (0, 1)$.
  If $f$ is bistable \ref{hyp:bistable}, there exists exactly one positive, bounded, strictly stable solution $u$ of \eqref{eq:main}.
  It satisfies $\partial_y u > 0$ and $u(x) \to 1$ uniformly as $\dist(x, \Omega^\cc) \to 1$.
  If $f$ is field-type \ref{hyp:field}, then \eqref{eq:main} does not admit a positive, bounded, strictly stable solution.
\end{theorem}
\noindent
This expands upon Theorem~\ref{thm:stable} from the introduction.

An examination of the proof shows that it readily extends to domains like circular cones that are uniformly Lipschitz and $\m{C}^{1,\al}$ off a finite set.
To simplify the presentation, we confine ourselves to the uniformly $\m{C}^{1,\al}$ case.
We only require this regularity to prove that a certain candidate solution is in fact strictly stable.
Several other conclusions in Theorem~\ref{thm:stable-expanded} hold in merely Lipschitz epigraphs.
We express these in lemmas below, which may be of independent interest.

To characterize the strictly stable solution, we begin with its limiting behavior.
\begin{lemma}
  \label{lem:stable-limit}
  Let $\Omega$ be a uniformly Lipschitz epigraph.
  If $f$ is bistable \ref{hyp:bistable}, every positive, bounded, strictly stable solution of \eqref{eq:main} converges to $1$ uniformly as $\dist(x, \Omega^\cc) \to 1$.
  If $f$ is field-type \ref{hyp:field}, then \eqref{eq:main} does not admit a positive, bounded, strictly stable solution.
\end{lemma}
\begin{proof}
  Suppose $f$ is either bistable or field-type and $u$ is a positive, bounded, \emph{weakly} stable solution of \eqref{eq:main}, so $\lambda(-\Delta - f'(u), \Omega) \geq 0$.
  Let $(x_n)_n$ be a sequence in $\Omega$ such that $\dist(x_n, \Omega^\cc) \to \infty$ as $n \to \infty$.
  Interior Schauder estimates allow us to extract a locally uniform subsequential limit $u_\infty$ of the shifted solutions $u(\anon + x_n)$ satisfying $-\Delta u_\infty = f(u_\infty)$ in $\R^d$.

  The limit $u_\infty$ is bounded, though not necessarily positive.
  By Lemma~\ref{lem:semicontinuous}, $\lambda(-\Delta - f'(u_\infty), \R^d) \geq 0$.
  That is, $u_\infty$ is weakly stable.
  With Liu, Wang, and Wu, the third author has shown that bounded weakly stable solutions of \eqref{eq:main} in $\R^d$ must be constant, provided all stable \emph{one-dimensional} solutions are constant (\mbox{\cite[Theorem~1.4]{LWWW}}).
  This holds in our case due to the unbalanced hypothesis in \ref{hyp:bistable}, which rules out monotone solutions in $\R$ (as does \ref{hyp:field}).
  The remaining nonconstant 1D solutions are either ground states or periodic, both unstable.
  
  If $f$ is field-type, the only weakly stable constant is $0$, so $u_\infty = 0$.
  This holds for every sequence $(x_n)_n$, so in fact $u(x) \to 0$ uniformly as $\dist(x, \Omega^\cc) \to \infty$ in $\Omega$.
  Because $u > 0$ vanishes on $\partial\Omega$ and as $y \to \infty$, $Q \coloneqq \{\partial_y u < 0\}$ is a proper subset of $\Omega$ satisfying $\partial_y u = 0$ on $\partial Q$.
  Using $\partial_y u$ as a witness in \eqref{eq:eigenvalue}, we see that $\lambda(-\Delta - f'(u), Q) \leq 0$.
  The principal eigenvalue is decreasing in the domain (\mbox{\cite[Proposition~2.3(iii)]{BR}}), so
  \begin{equation*}
    \lambda(-\Delta - f'(u),\Omega) \leq \lambda(-\Delta - f'(u), Q) \leq 0.
  \end{equation*}
  Hence $u$ is not strictly stable.
  (Indeed, it must be \emph{marginally} stable, with $\lambda = 0$.)
  
  If $f$ is bistable, the subsequential limit $u_\infty$ must be either $0$ or $1$ (the two stable roots of $f$).
  We claim that different sequences $(x_n)_n$ always produce the same limit.
  Otherwise, we could connect one sequence tending to $0$ with another tending to $1$ along contours far from $\partial\Omega$.
  Continuity would yield an intermediate sequence along which $u$ tends to $1/2$ at a point, which is forbidden by \cite[Theorem~1.4]{LWWW}.
  So $u$ has a genuine limit far from the boundary, and it is either $0$ or $1$.
  If $0$, the same argument given in the field case shows that $u$ is not strictly stable.
  We conclude that strictly stable solutions must converge to $1$ far from the boundary.
\end{proof}
We have now proved the field-type portion of Theorem~\ref{thm:stable-expanded} (which holds without $\m{C}^{1,\al}$ regularity), so in the remainder of the section we assume $f$ is bistable.
By Lemma~\ref{lem:stable-limit}, \emph{if} \eqref{eq:main} admits a strictly stable solution, it must converge to $1$ far from the boundary.
To conclude the proof, we show that there exists a solution converging to $1$, it is unique, and it is strictly stable.
These comprise the next three results.
\begin{lemma}
  \label{lem:existence}
  Let $\Omega$ be a uniformly Lipschitz epigraph and $f$ be bistable \ref{hyp:bistable}.
  Then there exists a positive bounded solution $u$ of \eqref{eq:main} converging to $1$ uniformly as $\dist(x, \Omega^\cc) \to 1$.
\end{lemma}
\begin{proof}
  Let $v$ solve the parabolic form of \eqref{eq:main} with initial data $1$:
  \begin{equation*}
    \begin{cases}
      \partial_t v = \Delta v + f(v) & \text{in } \Omega,\\
      v = 0 & \text{on } \partial\Omega,\\
      v(0, \anon) = 1 & \text{in } \Omega.
    \end{cases}
  \end{equation*}
  Because $1$ is a supersolution of \eqref{eq:main}, $\partial_t v < 0$, so $v$ has a long-time limit $u = v(\infty, \anon)$.
  Schauder estimates imply that $u$ is a nonnegative bounded solution of \eqref{eq:main}.

  Fix $s < 1$ such that $\int_0^s f(r) \d r > 0$, which exists by the unbalanced hypothesis in \ref{hyp:bistable}.
  In the proof of \cite[Proposition~2.2]{BG22}, the first two authors constructed a subsolution $0 \lneqq w \leq 1$ of $-\Delta v = f(v)$ supported in a ball of radius $R(s) > 0$ such that $w(0) = s$.
  If $\dist(x, \Omega^\cc) \geq R(s)$, $w_x \coloneqq w(\anon - x)$ is a subsolution of \eqref{eq:main} with $w_x(x) = s$.
  Since $v(0,\anon) = 1 > w_x$, the parabolic comparison principle implies that $v > w_x$ always.
  Sending $t \to \infty$, we see that $u \geq w_x$.
  So $u(x) \geq s > 0$, and by the strong maximum principle, $u$ is positive in $\Omega$.
  We can take $s$ arbitrarily close to $1$, so in fact $u(x) \to 1$ uniformly as $\dist(x, \Omega^\cc) \to \infty$, as desired.
\end{proof}
Now, uniqueness is a consequence of work of the first author with Caffarelli and Nirenberg~\cite{BCN97b}.
\begin{lemma}
  \label{lem:epigraph-uniqueness}
  Let $\Omega$ be a uniformly Lipschitz epigraph and $f$ be bistable \ref{hyp:bistable}.
  Then \eqref{eq:main} admits a unique positive bounded solution that converges to $1$ uniformly as $\dist(x, \Omega^\cc) \to 1$.
  Moreover, $\partial_y u > 0$.
\end{lemma}
\begin{proof}
  This follows immediately from the proof of Theorem~1.2(d) in \cite{BCN97b} given in Section 5 of the same.
  The argument only requires a standard form of the maximum principle and the uniform convergence of solutions to $1$ at infinity.
  In \cite{BCN97b}, the authors assume that the nonlinearity $f$ is monostable rather than bistable, meaning $f > 0$ in $(0, 1)$.
  This implies that all positive bounded solutions converge to $1$ at infinity, so the uniqueness result in \cite{BCN97b} is unconditional.
  For our bistable nonlinearities, we instead assume such convergence as a hypothesis to obtain conditional uniqueness.
\end{proof}
Now, we need only show that $u$ is strictly stable.
\begin{lemma}
  \label{lem:strictly-stable}
  Suppose $\Omega$ is the epigraph of a uniformly $\m{C}^{1,\al}$ function for some $\al \in (0, 1)$.
  Let $f$ be bistable \ref{hyp:bistable} and $u$ be the unique positive bounded solution of \eqref{eq:main} converging to $1$ uniformly as $\dist(x, \Omega^\cc) \to 1$.
  Then $u$ is strictly stable: $\lambda(-\Delta - f'(u), \Omega) > 0$.
\end{lemma}
\begin{proof}
  We follow the proof of Proposition~4.3 in~\cite{BG23b}.
  Again, that work assumed that $f$ is monostable and concluded that solutions of \eqref{eq:main} converge to $1$ far from the boundary.
  Here, we take this convergence as a hypothesis.
  We are working with $\m{C}^{1,\al}$ rather than $\m{C}^{2,\al}$ boundary, so we verify that the proof of \cite[Proposition~4.3]{BG23b} continues to hold in our setting.
  We summarize except where boundary regularity comes into play.

  Let $y' \coloneqq y - \phi(x)$, so $\Omega = \{y > \phi(x')\} = \{y' > 0\}$.
  Because $f \in \m{C}^1$ and $f'(1) < 0$, there exists $\eta > 0$ such that $f' \leq \frac{2}{3} f'(1) < 0$ on $(\eta, 1)$.
  By hypothesis, $u$ tends to $1$ uniformly away from the boundary.
  Hence $\op{Lip}\phi < \infty$ implies that there exists $H \geq 2$ such that $u > \eta$ where $y' > H$, and in particular
  \begin{equation}
    \label{eq:decreasing}
    -f'(u) \geq \frac{2}{3} \abs{f'(1)} > 0 \quad \text{on } \{y' > H\}.
  \end{equation}
  Let $\m{L} \coloneqq \Delta + f'(u)$ and write $u' \coloneqq \partial_y u$, which satisfies $\m{L} u' = 0$ and $u' > 0$ due to Lemma~\ref{lem:epigraph-uniqueness}.

  Given $R \geq H + 1$, let $G_R \coloneqq \{(x',y') : \abs{x'} < R, 0 < y' < R\}$ and $\lambda_R \coloneqq \lambda(-\m{L}, G_R).$
  By \cite[Proposition~2.3(iv)]{BR}, $\lambda(-\m{L},\Omega) = \lim_{R \to \infty} \lambda_R$.
  The eigenvalues $\lambda_R$ are nonincreasing in $R$ because $(G_R)_R$ is increasing, and we wish to show that they remain uniformly positive as $R \to \infty$.
  We are done if $\lambda_R \geq \frac{1}{3} \abs{f'(1)}$, so suppose there exists $\ubar{R} \geq 3H$ such that $\lambda_R \leq \frac{1}{3}\abs{f'(1)}$ for all $R \geq \ubar{R}$.
  From now on, we assume this condition on $R$.

  For each $R$, we have a corresponding positive principal eigenfunction $\psi_R$ satisfying $-\m{L} \psi_R = \lambda_R \psi_R$.
  By Theorem~8.33 of \cite{GT}, $u$ and $\psi_R$ are $\m{C}^{1,\al}$ on $\bar\Omega$.
  Using an intermediate Schauder estimate such as \cite[Theorem~5.1]{GH}, we can integrate $-\m{L} \psi_R = \lambda_R \psi_R$ and $\m{L}u' = 0$ by parts to obtain
  \begin{equation*}
    \lambda_R \int_{G_R} u' \psi_R = \int_{\partial G_R} u' \abss{\partial_\nu \psi_R}.
  \end{equation*}
  Throughout, we use $\nu$ to denotes the outward normal vector field on the domain of integration.
  Let $\partial_- G_R \coloneqq \partial G_R \cap \partial \Omega$ denote the bottom portion of $\partial G_R$.
  Then
  \begin{equation*}
    \lambda_R \geq \frac{\int_{\partial_- G_R}u' \abss{\partial_\nu \psi_R}}{\int_{G_R} u' \psi_R}.
  \end{equation*}

  Now let $S_R \coloneqq G_R \cap \{y' < H\}$.
  To begin, we show that $\int_{G_R} u' \psi_R \lesssim \int_{S_R} u' \psi_R$.
  Let $U_R \coloneqq G_R \cap \{y' > H\}$ and write $\partial_+ S_R = G_R \cap \{y' = H\}$.
  Using \eqref{eq:decreasing} and $\lambda_R \leq \frac{1}{3} \abs{f'(1)}$, integration by parts and routine inequalities yield
  \begin{equation*}
    \int_{U_R} \psi_R \lesssim \int_{\partial_+S_R} \abs{\partial_\nu \psi_R}.
  \end{equation*}
  Consider $(x',H) \in \partial_+ S_R$ with $x' \in B_{R-1}'$, where we use $B'$ to indicate balls in $\R^{d-1}$.
  Interior Schauder estimates and the Harnack inequality imply that
  \begin{equation*}
    \abs{\nab \psi_R(x',H)} \lesssim \int_{H-1}^H \psi_R(x', y'),
  \end{equation*}
  so
  \begin{equation*}
    \int_{B_{R-1}'} \abs{\partial_\nu \psi_R(x',H)} \d x' \lesssim \int_{B_{R-1}' \times (H-1,H)} \psi_R(x', y') \d x' \ds y'.
  \end{equation*}
  Schauder estimates and the boundary Harnack inequality (Theorem 1.1 of \cite{CFMS}) allow us to extend the radius $R-1$ to $R$, so
  \begin{equation*}
    \int_{U_R} \psi_R \lesssim \int_{\partial_+S_R} \abs{\partial_\nu \psi_R} \lesssim \int_{S_R \cap \{H-1 < y' < H\}} \psi_R.
  \end{equation*}
  Finally, $u \geq \eta > 0$ where $\{y' = H\}$ while $u = 0$ on $\{y' = 0\}$.
  By the mean value theorem and \cite[Theorem 1.1]{CFMS}, $u' \gtrsim 1$ in the strip $\{H-1 < y' < H\}$.
  Therefore
  \begin{equation*}
    \int_{U_R} u' \psi_R \lesssim \int_{S_R \cap \{H-1 < y' < H\}} u' \psi_R \leq \int_{S_R} u' \psi_R.
  \end{equation*}
  Because $G_R = U_R \cup S_R$ (up to closure), we see that $\int_{G_R} u' \psi_R \lesssim \int_{S_R} u' \psi_R$ and hence
  \begin{equation}
    \label{eq:reduced}
    \lambda_R \gtrsim \frac{\int_{\partial_- G_R}u' \abss{\partial_\nu \psi_R}}{\int_{S_R} u' \psi_R}.
  \end{equation}
  Thus we have reduced our analysis to a band along the boundary of $\Omega$.

  For each $x' \in B_{R - 2H}'$, \cite[Theorem 1.1]{CFMS} yields
  \begin{equation*}
    \int_0^H u'(x', y') \psi_R(x', y') \d y' \lesssim \psi_R(x', H).
  \end{equation*}
  Applying the same boundary Harnack inequality in the lateral region $S_R \setminus S_{R - 2H}$, we can check that
  \begin{equation}
    \label{eq:shrink}
    \int_{S_R} u' \psi_R \lesssim \int_{\partial_+ S_{R - 2H}} \psi_R.
  \end{equation}
  Now fix $x' \in B_{R - 2H}'$.
  We claim that
  \begin{equation}
    \label{eq:Hopf}
    u'(x', y') \gtrsim 1 \And \psi_R(x',y') \lesssim \partial_y\psi_R(x',0) y' 
  \end{equation}
  for all $y' \in [0, H].$
  In our judgment, this is the most delicate part of the argument, so we reproduce the proof in full.
  Toward a contradiction, suppose there exists a sequence of radii $R_n \geq \ubar{R} \geq 3H$ and points $x_n = (x_n', y_n') \in S_{R_n - 2H}$ such that
  \begin{equation*}
    \partial_y \psi_{R_n}(x_n', 0) \leq \frac{1}{n y_n'} \psi_{R_n}(x_n', y_n').
  \end{equation*}
  Writing $A_n(x) \coloneqq y_n' x + x_n$ and
  \begin{equation*}
    \Psi_n \coloneqq \frac{\psi_{R_n} \circ A_n}{\psi_{R_n} \circ A_n(0)},
  \end{equation*}
  we have
  \begin{equation*}
    \begin{cases}
      -\Delta \Psi_n = (y_n')^2 (f' \circ u \circ A_n + \lambda_{R_n}) \Psi_n & \text{in } A_n^{-1} G_{R_n},\\
      \Psi_n = 0 & \text{on } \partial A_n^{-1} G_{R_n}
    \end{cases}
  \end{equation*}
  as well as $\Psi_n(0) = 1$ and $\partial_y \Psi_n(0, -1) \leq n^{-1}$.
  Because $y_n' \leq H$, we have isotropically dilated our domain by a factor of at least $H^{-1}$.
  We assumed $x_n \in S_{R_n - 2H}$, so $B_2 \cap \partial A_n^{-1} G_{R_n} = B_2 \cap \partial A_n^{-1} \Omega$.
  That is, the origin is at least distance $2$ from the corners of $A_n^{-1}G_{R_n}$.

  Moreover, the sequence of epigraphs $A_n^{-1} G_{R_n}$ is uniformly $\m{C}^{1,\al}$, as are the coefficients in the elliptic equation for $\Psi_n$.
  Hence $\Psi_n$ obeys uniform Harnack and Schauder estimates up to the boundary, meaning $\Psi_n$ is likewise uniformly $\m{C}^{1,\al}$.
  We can thus extract locally uniform subsequential limits $\Psi_\infty$ of $(\Psi_n)_n$, $\Omega_\infty$ of $(A_n^{-1} \Omega)_n$, and $V_\infty$ of $(y_n')^2 (f' \circ u \circ A_n + \lambda_{R_n})$.
  Here we use the fact that $\lambda_{R_n} \lesssim \op{Lip}f + 1$ is uniformly bounded because $A_n^{-1} G_{R_n}$ contains a ball of some size $c > 0$ independent of $n$.
  The limits $\Psi_\infty$, $\Omega_\infty$, and $V_\infty$ are all uniformly $\m{C}^{1,\al}$ and
  \begin{equation*}
    \begin{cases}
      -\Delta \Psi_\infty = V_\infty \Psi_\infty & \text{in } B_2 \cap \Omega_\infty,\\
      \Psi_\infty = 0 & \text{on } B_2 \cap \partial \Omega_\infty
    \end{cases}
  \end{equation*}
  while $\Psi_\infty(0) = 1$ and $\partial_y \Psi_\infty(0, -1) = 0$.
  This, however, contradicts the Hopf lemma, which holds on $\m{C}^{1,\al}$ domains~\cite{Giraud,Lieberman}.
  The contradiction proves the second part of \eqref{eq:Hopf}.
  The first follows from similar (in fact, simpler) considerations.
  
  To conclude, we observe that $\partial_y \psi_R \asymp \abs{\partial_\nu \psi_R}$ on $\partial\Omega$ because $\phi$ is uniformly Lipschitz.
  So \eqref{eq:shrink} and \eqref{eq:Hopf} yield
  \begin{equation*}
    \int_{S_R} u' \psi_r \lesssim \int_{\partial_+ S_{R - 2H}} \psi_R \lesssim \int_{\partial_- G_R} u' \abs{\partial_\nu \psi_R}.
  \end{equation*}
  By \eqref{eq:reduced}, $\lambda_R \gtrsim 1$, as desired.
\end{proof}

\subsection{Instability of spike solutions}
\label{sec:unstable}

As a partial converse to the above results, we show that the ``spike solutions'' constructed in Section~\ref{sec:multiple} are strictly unstable.
\begin{proposition}
  \label{prop:unstable}
  Let $\Omega$ be a locally Lipschitz domain (not necessarily an epigraph) with aperture greater than $\pi$, and let $f$ be bistable \ref{hyp:bistable} or field-type \ref{hyp:field} satisfying \ref{hyp:ground}.
  Then the solutions of \eqref{eq:main} constructed in Theorems~\ref{thm:multiple} and \ref{thm:field-multiple} are strictly unstable.
\end{proposition}
\begin{proof}
  Let $u$ be a spike solution constructed in Section~\ref{sec:multiple}.
  It consists of a bent chain of widely separated ground states.
  Sending $x \cdot \theta_\pm \to \infty$ along the chain, \eqref{eq:Delaunay-limit} states that $u$ converges locally uniformly to the periodic Delaunay solution $u_{L_\pm}$ from Proposition~\ref{prop:Delaunay} for some $L_\pm \gg 1$.
  We arbitrarily choose $x \cdot \theta_+ \to \infty$ and write $L = L_+$.
  By Lemma~\ref{lem:semicontinuous}, $\lambda(-\Delta - f'(u), \Omega) \leq \lambda(-\Delta - f'(u_L), \R^d)$.
  Thus it suffices to show that $u_L$ is strictly unstable.

  We work on the fundamental domain $D_L = \{\abs{y} \leq L/2\}$, which we identify with $\R^{d-1} \times L \T$.
  Because $u_L$ is even in $y$, $\partial_y u_L$ is a sign-changing null eigenfunction of $-\Delta + f'(u_L)$ in $D_L$.
  Moreover, $u_L \to 0$ at infinity, which implies that $-f'(u_L) \geq \abs{f'(0)}/2 > 0$ outside a compact subset of $D_L$.
  It follows that $-\Delta - f'(u_L)$ has pure point spectrum below $\abs{f'(0)}/2$ (for example, \emph{Cf.} Theorem 9.38 in \cite{Teschl}), and the principal eigenvalue is simple.
  Because $\partial_y u_L$ changes sign, it is not principal.
  Hence $\lambda(-\Delta - f'(u_L), \R^d) < 0$, as desired.
\end{proof}
If $\Omega$ is an epigraph with aperture exceeding $\pi$, we have identified a single strictly stable solution and a wealth of strictly unstable spike solutions.
We conjecture that \emph{all} remaining solutions are strictly unstable, but we do not pursue the matter further.

\appendix

\section{Nondegeneracy}
\label{sec:nondegenerate}

In this appendix, we gather a number of existing results regarding the nondegeneracy \ref{hyp:ground} of ground states.
The existence of ground states has been widely studied for a range of nonlinearities.
For example, in collaboration with Lions and Peletier, the first author showed the following:
\begin{proposition}[\cite{BLP}]
  \label{prop:ground-existence}
  If $f$ satisfies \ref{hyp:bistable}, then it admits a ground state.
  If $f$ satisfies \ref{hyp:field} and $f(u) \ll u^{\frac{d+2}{d-2}}$ as $u \to \infty$ when $d \geq 3$ or $f(u) \ll u^p$ for some $p < \infty$ when $d = 2$, then a ground state exists.
\end{proposition}
\noindent
We note that the growth condition in the field case is rather sharp.
Pohozaev~\cite{Pohozaev} showed that $f(u) = -u + u^p$ has no ground state when $d \geq 3$ and $p \geq \frac{d + 2}{d - 2}$.

The nondegeneracy \eqref{eq:nondegenerate} in \ref{hyp:ground} is more subtle.
While we expect this condition to be generic, it does not hold universally, and can be difficult to verify for a given nonlinearity.
For the reader's convenience, we therefore collect several conditions ensuring \eqref{eq:nondegenerate} that have appeared in the literature.

By~\cite{GNN79}, ground states are radially symmetric.
In Lemma~4.2 of \cite{NT}, Ni and Takagi showed that the multidimensional statement \eqref{eq:nondegenerate} is equivalent to the \emph{radial} stability of $U$.
\begin{lemma}[\cite{NT}]
  \label{lem:radial}
  If $U$ is stable to radial perturbations in the sense that the operator $-\partial_r^2 - \tfrac{d-1}{r} \partial_r - (f' \circ U)(r)$ has a bounded inverse on $L^2(\R)$, then \eqref{eq:nondegenerate} holds.
\end{lemma}
\noindent
Thus nondegeneracy is an essentially one-dimensional phenomenon.
This opens the door to ODE arguments, which feature in the proofs of the following.
\begin{proposition}[\cite{Kwong,PS,WW,NTW}]
  \label{prop:nondegenerate}
  Every ground state is nondegenerate, so \ref{hyp:ground} holds, under either of the following conditions:
  \begin{enumerate}[label = \textup{(\roman*)}]
  \item $f(u) = -u + u^p$ for $p \in (1, \frac{d+2}{d-2})$ when $d \geq 3$ or $p \in (1, \infty)$ when $d = 2$;
    \label{item:power}

  \item $f$ satisfies \ref{hyp:bistable} and the function $u \mapsto \frac{f(u)}{u-a}$ is nonincreasing in $(a, 1)$, where $a$ is the unique value in $(\theta, 1)$ at which $\int_0^a f = 0$.
    \label{item:strong}
  \end{enumerate}
\end{proposition}
\noindent
Nondegeneracy is often closely tied to the uniqueness of the ground state $U$.
Thus Proposition~\ref{prop:nondegenerate} was essentially shown by Kwong~\cite{Kwong} and Peletier and Serrin~\cite{PS}, but not stated explicitly.
The third author has explicitly shown nondegeneracy in~\cite[Lemma~13.4]{WW} with Winter and~\cite[Lemma~5.2]{NTW} with Ni and Takagi.

Conversely, constructions of multiple ground states implicitly lead to examples of degenerate states.
For example, Peletier and Serrin exhibit a bistable nonlinearity admitting multiple ground states~\cite[\S5]{PS}.
Interpolating between this example and a bistability with uniqueness, we should observe multiple solutions emerging out of a bifurcation.
We expect the ground state at the bifurcation to be degenerate.

Finally, we observe that \ref{item:strong} implies nondegeneracy for the model bistability $f(u) = u(u - \theta)(1 - u)$ for $\theta \in (0, 1/2)$, which ensures $\int_0^1 f > 0$.

\printbibliography
\end{document}